\documentclass[a4paper,10pt]{article}

\title{Cycles of links and fixed points for orientation preserving homeomorphisms of the open unit disk}
\author{Juliana Xavier}
\date{}

\usepackage{psfrag}
\usepackage{fancyhdr}
\usepackage{mathrsfs}
\usepackage{verbatim}
\usepackage[ansinew]{inputenc}
\usepackage{amsmath,amsthm,amssymb,amscd}
\usepackage[all]{xy}
\usepackage{xspace}
\usepackage[dvips]{color}
\usepackage{epsfig}
\usepackage{makeidx}
\usepackage{pb-diagram}
\usepackage{amsfonts}
\usepackage{graphicx}
\usepackage{subfigure}
\usepackage[absolute]{textpos}

\newcommand{\Z}{\mathbb{Z}}

\newcommand{\R}{\mathbb{R}}
\newcommand{\C}{\mathbb{C}}
\newcommand{\D}{\mathbb{D}}
\newcommand{\om}{\omega}

\newcommand{\toe}{\stackrel{=}{\longrightarrow}}

\DeclareMathOperator{\fix}{Fix}
\DeclareMathOperator{\per}{Per}
\DeclareMathOperator{\homeo}{Homeo}
\DeclareMathOperator{\inte}{Int}
\DeclareMathOperator{\id}{Id}

\newtheorem{teo}{Theorem}[section]
\newtheorem{cor}[teo]{Corollary}
\newtheorem{lema}[teo]{Lemma}
\newtheorem{prop}[teo]{Proposition}

\theoremstyle{definition}

\newtheorem{obs}[teo]{Remark}

\theoremstyle{remark}
\makeindex

\begin{document}

\maketitle

\begin{abstract} Michael Handel proved in \cite{handel} the existence
of a fixed point for an orientation preserving homeomorphism of the
open unit disk that can be extended to the closed disk, provided that
it has points whose orbits form an {\it oriented cycle of links at
infinity}.  More recently, the author generalized Handel's theorem to a wider class of cycles of links \cite{joul}. 
In this paper we complete this topic describing exactly which are all the cycles of links forcing the existence of a
fixed point. 

\end{abstract}

\section{Introduction}\label{intro}

Handel's fixed point theorem \cite{handel} has been of great
importance for the study of surface homeomorphisms.  It guarantees
the existence of a fixed point for an orientation preserving
homeomorphism $f$ of the unit disk $\D = \{z\in \C : |z| < 1\}$
provided that it can be extended to the boundary $S^1 = \{z\in \C :
|z|=1\}$ and that it has points whose orbits form an oriented cycle
of links at infinity.  More precisely, there exist $n$ points $z_i
\in \D$ such that

$$\lim _{k \to -\infty} f^k(z_i) = \alpha _i \in S^1 ,  \ \lim _{k
\to +\infty} f^k(z_i) = \om _i \in S^1 ,$$

\noindent $i=1, \ldots, n$, where the $2n$ points $\{\alpha _i\}$, $\{\om _i\}$ are different points in $S^1$
and satisfy the following order property:

(*) $\alpha_{i+1}$ is the only one among these points that lies in the  open interval in the oriented circle $S^1$
from $\om_{i-1}$ to $\om_i$ .

\noindent (Although this is not Handel's original statement, it is an equivalent one as already pointed out in
\cite{patrice}).

Le Calvez gave an alternative proof of this theorem \cite{patrice},
relying only in Brouwer theory and plane topology, which allowed him
to obtain a sharper result. Namely, he weakened the extension
hypothesis by demanding  the homeomorphism to be extended just to $\D
\cup (\cup _{i \in \Z/n\Z} \{\alpha _i , \om _i\})$ and he strengthed
the conclusion by proving the existence of a simple closed curve of
index 1.

The author generalized both Handel's and Le Calvez's results as follows \cite{joul}. Let $P \subset \D$ be a compact convex $n$-gon.  Let $\{v_i : i\in \Z/n\Z\}$ be its set of vertices and for each
$i\in \Z/n\Z$, let  $e_i$ be the edge joining
$v_i$ and $v_{i+1}$. We suppose that each $e_i$ is endowed with an
orientation, so that we can tell whether $P$ is to the right or to the left of $e_i$ .  We say that the orientations of
$e_i$ and $e_j$
{\it coincide} if $P$ is to the right (or to the left) of both $e_i$ and $e_j$, $i, j \in \Z/n\Z$. \\
We define the {\it index} of $P$ by

$$i(P) = 1 - \frac{1}{2} \sum _{i \in \Z/n\Z} \delta _i,$$

\noindent where $\delta _i = 0$ if the orientations of $e_{i-1}$ and $e_i$
coincide, and $\delta _i = 1$ otherwise.

We will note $\alpha _i$ and $\om _i$ the first, and respectively the last,
point  where the straight line $\Delta _i$ containing $e_i$ and inheriting its
orientation intersects $\partial \D$.

\begin{figure}[h]
\begin{center}

   \subfigure[Handel's index 1 polygon]{\includegraphics[scale=0.5]{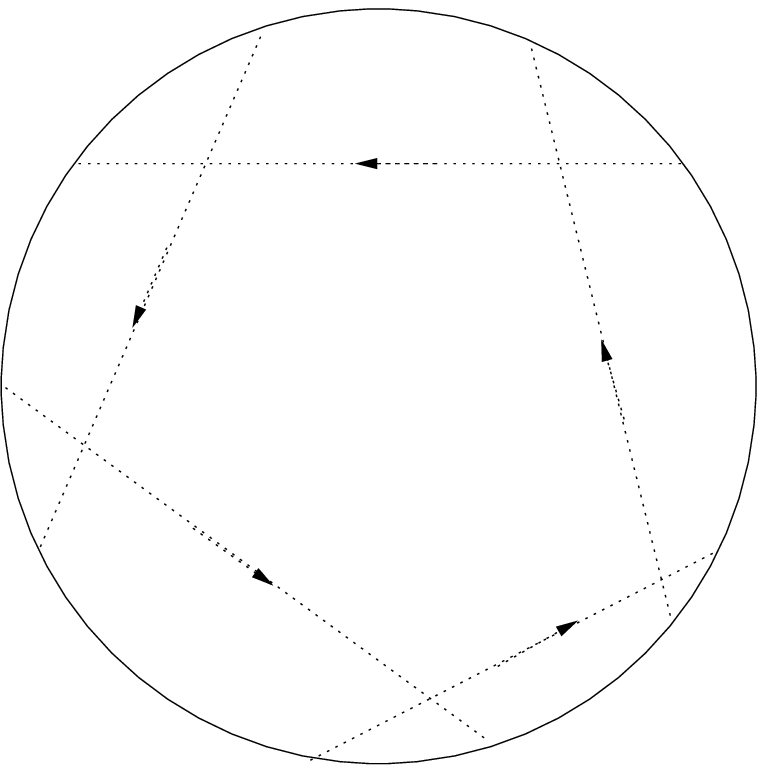}}\hspace{.25in}
    \subfigure[ Index -1 polygon] {\includegraphics[scale=0.5]{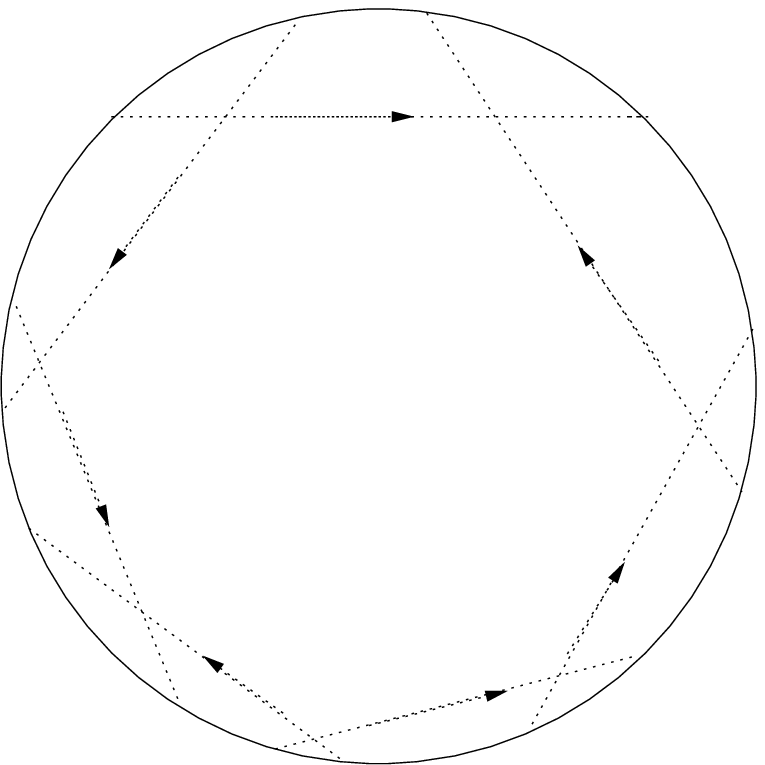}}\hspace{.25in}
    \subfigure[$\om _i = \alpha _{i+2} \ \forall i$ ] {\includegraphics[scale=0.5]{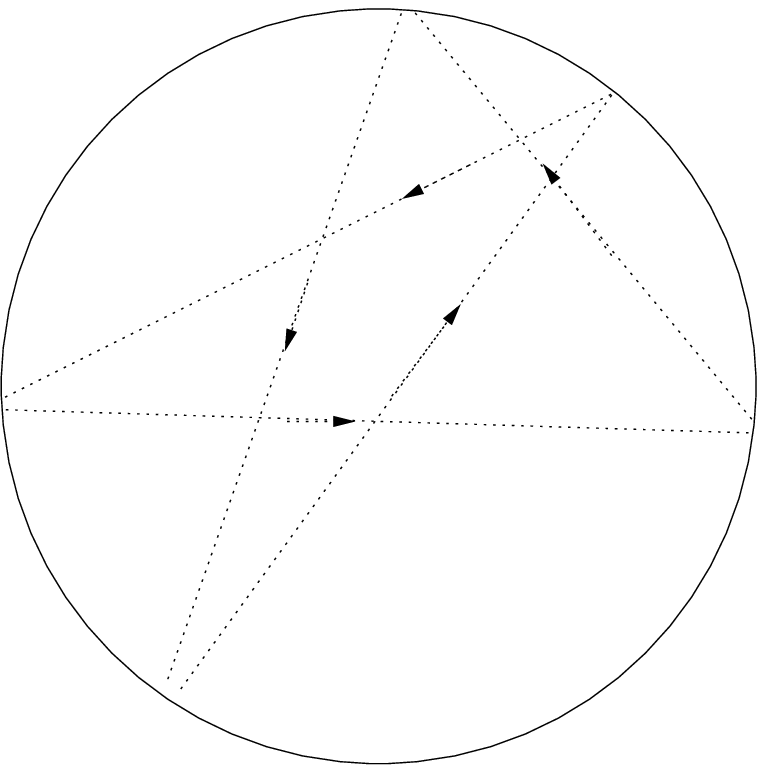}} \\
\caption {The hypothesis of Theorem \ref{poli}.}
\end{center}
 \end{figure}

\noindent We say that a homeomorphism $f:\D \to \D$ {\it realizes} $P$ if
there exists a family $(z_i)_{i\in \Z/n\Z}$ of points in $\D$ such
that for all $i\in \Z/n\Z$,
$$\lim _{k \to -\infty} f^k(z_i) = \alpha _i ,  \ \lim _{k \to
+\infty} f^k(z_i) = \om _i .$$\\

\begin{teo}\label{poli}{\bf \cite{joul}} Let $f : \D \to \D$ be an
orientation preserving homeomorphism  which realizes a compact convex
polygon $P\subset \D$ where the points $\alpha_i, \om_i, i\in \Z/n\Z$ are all different.  Suppose that $f$ can be extended to a homeomorphism of $\D
\cup (\cup _{i \in \Z/n\Z} \{\alpha _i , \om _i\}).$\\
If $i(P) \neq 0$, then $f$ has a fixed point.  Furthermore, if
$i(P)=1$, then there exists a simple closed curve $C\subset \D$ of
index  1.

\end{teo}

The two polygons appearing in Figure 1 (a) and (b) satisfy the hypothesis of this theorem.  However,  the polygon
illustrated in (c) does not, as there are coincidences among the points
$\{\alpha_i\}, \{\om_i\}$, $i\in \Z/n\Z$.\\

The purpose of this paper is to complete this topic: we assume that there exists a family $(z_i)_{i\in \Z/n\Z}$ of points in $\D$ and two families $(\alpha_i)_{i\in
\Z/n\Z} , (\om_i)_{i\in\Z/n\Z}$ of points in $S^1$ such
that for all $i\in \Z/n\Z$,
$$\lim _{k \to -\infty} f^k(z_i) = \alpha _i ,  \ \lim _{k \to
+\infty} f^k(z_i) = \om _i ,$$\noindent that the homeomorphism $f$ extends to a homeomorphism of $\D
\cup (\cup _{i \in \Z/n\Z} \{\alpha _i , \om _i\}),\ $\noindent and describe exactly which combinatorics of the
points $\alpha_i, \om_i, i\in \Z/n\Z$ force the existence of a fixed point.

A {\it cycle of links of order $n\geq 3$} is a family of pairs of points on the circle $S^1$,
$${\cal L}= ((\alpha_i, \om_i))_{ i\in\Z/n\Z}$$ \noindent  such that for all
$i\in\Z/n\Z$:

\begin{enumerate}
 \item $\alpha_i \neq \om_i$,
\item  $\alpha_{i+1}$ and  $\om_{i+1}$ belong to different connected components of $S^1\backslash \{\alpha_i, \om_i\}$.
\end{enumerate}

If ${\cal L}$ is  a cycle of links, we define the set
$$\ell = \{\alpha_i, \om_i: i\in\Z/n\Z\}\subset S^1$$\noindent of points in the circle which belong to a pair in the cycle.

If $a,b\in \ell$, we note $a\to b$ if $b$ follows $a$ in the natural (positive) cyclic order on $S^1$, and $a\toe b$ if
either $a=b$ or $a\to b$.\\

\noindent We say that a cycle of links ${\cal L}$ is {\it elliptic}\index{cycle of links! elliptic} if for all $i\in\Z/n\Z$:

$$ \om_{i-1}\toe \alpha_{i+1}\to \om_{i}.$$

\noindent We say it is {\it hyperbolic} if $n=2k, k\geq 2$ and for all $i\in\Z/n\Z$, $i=0\mod 2$:

$$\alpha_{i}\to \alpha_{i-1}\toe \om_{i+1}\to \om_{i}\toe \alpha_{i+2}.$$\\

\begin{figure}[h]
\begin{center}
\psfrag{0}{$\alpha _{0}$}\psfrag{1}{$\om_{2} $}\psfrag{2}{$\alpha _1$}
\psfrag{3}{$ \om _{0}$}\psfrag{4}{$\alpha _{2}$}\psfrag{5}{$\om_{1}$}
\psfrag{00}{$\alpha _{0}$}\psfrag{11}{$\alpha_{3} $}\psfrag{22}{$\om _1$}
\psfrag{33}{$ \om _{0}=\alpha_2$}\psfrag{55}{$\alpha_{1}$}\psfrag{66}{$\om_{3}$}\psfrag{77}{$\om_{2}$}
   \subfigure[An elliptic cycle of links of order 3]{\includegraphics[scale=0.5]{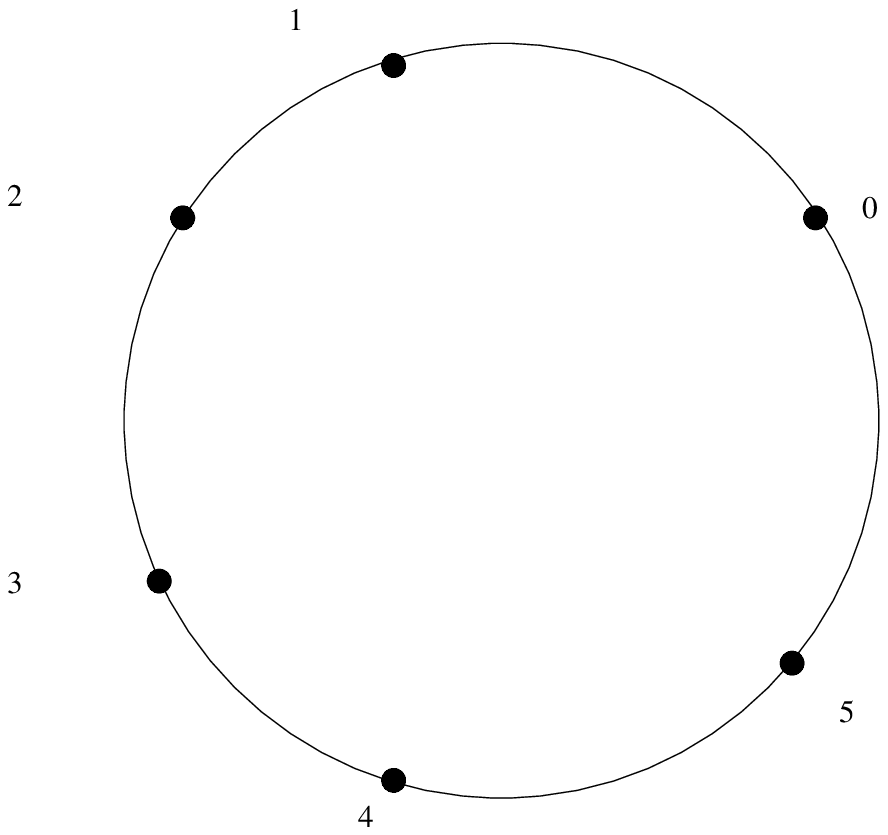}}\hspace{.25in}
    \subfigure[A hyperbolic cycle of links of order 4] {\includegraphics[scale=0.5]{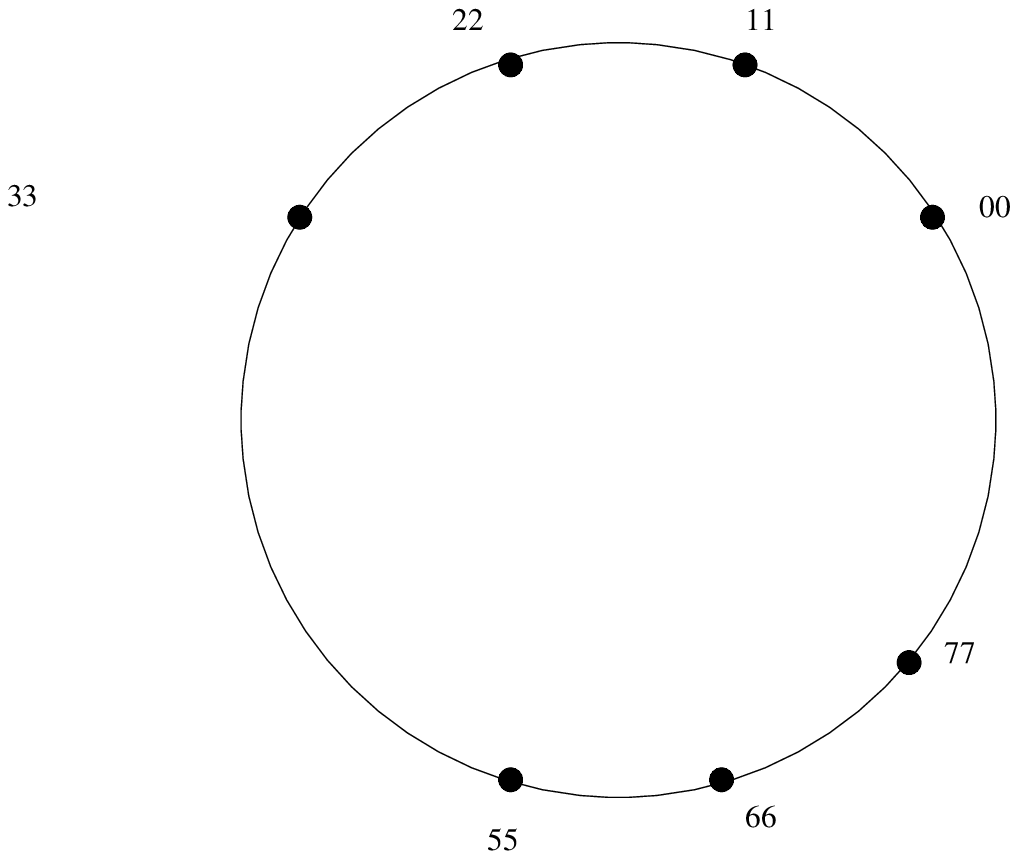}}\hspace{.25in}
     \\

\end{center}
 \end{figure}

\noindent We say that ${\cal L}$ is {\it non-degenerate} if:
$$(\alpha_i, \om_i)\in {\cal L} \Rightarrow ( \om_i, \alpha_i)\notin {\cal L} .$$ \noindent Of course, we say it is
{\it degenerate}, if this condition is not satisfied.  An example is illustrated in Figure 2. \\

\begin{figure}[h]\label{ptpt}
\begin{center}
\psfrag{0}{$\alpha _{0}=\om_2$}\psfrag{c}{$\om_0=\alpha_2 $}\psfrag{1}{$\alpha_1 = \om_3$}
\psfrag{u}{$ \om _{1}=\alpha_3$}
\includegraphics[scale=0.5]{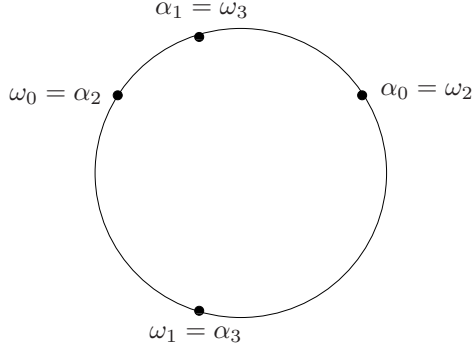}
\caption{A degenerate cycle of links }
\end{center}
\end{figure}

\noindent We say that a homeomorphism $f:\D \to \D$ {\it realizes} ${\cal L}$ if
there exists a family $(z_i)_{i\in \Z/n\Z}$ of points in $\D$ such
that for all $i\in \Z/n\Z$,
$$\lim _{k \to -\infty} f^k(z_i) = \alpha _i ,  \ \lim _{k \to
+\infty} f^k(z_i) = \om _i .$$\\

The following result is the main theorem of this article.

\begin{teo}\label{main*}  Suppose that $f : \D \to \D$ is an
orientation preserving homeomorphism  which realizes a cycle of links ${\cal L}$ and can be extended to a
homeomorphism of $\D
\cup \ell .$\\
If ${\cal L}$ is either elliptic or hyperbolic, then $f$ has a fixed point.  Furthermore, if ${\cal L}$ is non-degenerate and
elliptic, then there exists a simple closed curve $C\subset \D$ of
index  1 .
\end{teo}

It turns out that these results completely describe the combinatorics giving rise to fixed points:

\begin{lema}\label{opt}  Given a
family $((\alpha_i, \om_i))_{i\in\Z/n\Z}$ of pairs of points in $S^1$, then one of the following is true:

\begin{enumerate}
\item there exists a subfamily of $((\alpha_i, \om_i))_{i\in\Z/n\Z}$ forming an elliptic or hyperbolic cycle of links,
\item  the straight oriented lines from $\alpha_i$ to $\om_i$ bound a non-zero index polygon $P\subset \D$,
\item there exists a {\it fixed-point free} orientation preserving homeomorphism $f:\D\to \D$, and a family of
points $(z_i)_{i\in\Z/n\Z}$ in $\D$ such that for all $i\in \Z/n\Z$,
$$\lim _{k \to -\infty} f^k(z_i) = \alpha _i ,  \ \lim _{k \to
+\infty} f^k(z_i) = \om _i. $$
\end{enumerate}
\end{lema}
\newpage

We finish this introduction with some remarks on Theorem \ref{main*}.

\subsubsection*{The elliptic non-degenerate case contains Le Calvez's improvement of
Handel's theorem.}  Indeed, if the points in $\ell$ are all different, ${\cal L}$ is non-degenerate.  As the example in Figure 1 (c)
shows,
our theorem is more general even in this case.

\subsubsection* {The theorem contains the author's result on non-zero index polygons.} Indeed, in \cite{joul} it is shown that if $f$ realizes a non-zero index polygon where
 the points
 $\alpha_i, \om_i, i\in \Z/n\Z$ are all different, then $f$ realizes an elliptic or hyperbolic cycle of links. Again, as coincidences in $\ell$ are allowed, our
theorem is more general even in this case.

\subsubsection*{The extension hypothesis is needed. } Indeed,  if $f:\D\to \D$ is fixed-point free,  one
can easily construct a homeomorphism $h:\D \to \D$ such that $hfh^{-1}$ realizes any prescribed cycle of
links.

\subsubsection* {Non-degeneracy is needed for obtaining the index result.}

 Let $f_1$ be the time-one map of the flow whose orbits
are drawn in the figure below.

\begin{figure}[h]
\begin{center}
\psfrag{b}{$\alpha_1=\om _3$}\psfrag{c}{$\alpha_2 = \om_0$}
\psfrag{d}{$\alpha_3 = \om_1$}\psfrag{e}{$\alpha_0 = \om_2$}\psfrag{x}{$x$}
\includegraphics[scale=0.6]{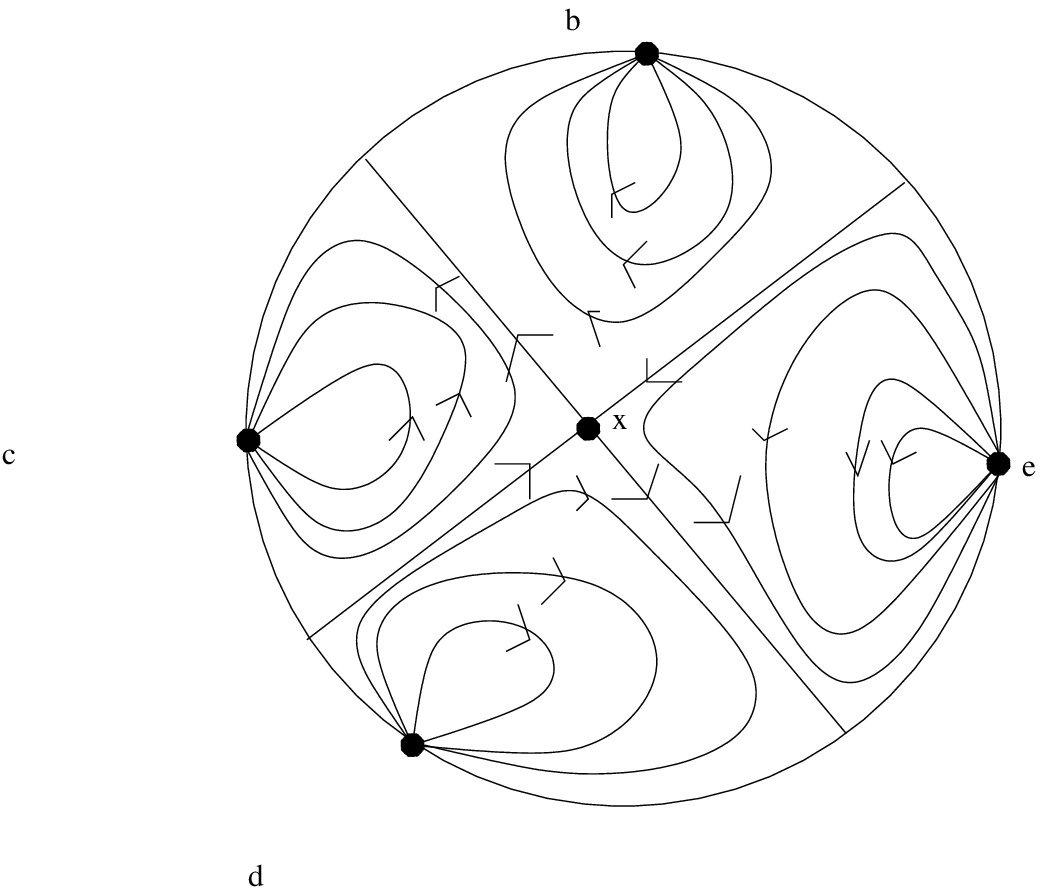}
\end{center}
\end{figure}

As we will explain below, one can perturb $f_1$ in a homeomorphism $f$ such that:

\begin{itemize}
\item $\fix(f) = \fix(f_1) = \{x\}$,
\item $f = f_1$ in a neighbourhood of $x$,
\item $f$ realizes ${\cal L} = ((\alpha_i, \om_i))_{i\in\Z/4\Z}$.
\end{itemize}

We say that the set $X$ is {\it free} if $f(X)\cap X = \emptyset$.

One can find (by means of a transverse foliation, for example),  free and pairwise disjoint simple paths $\beta _i$
and $\gamma _i$, $i\in\Z/4\Z$ such that :

\begin{itemize}
 \item  $\beta_i$ joins $z_i$ and $z_i^{'}$, where $\lim _{k\to \infty} f_1^{-k} (z_i) = \alpha_i$ and
$\lim _{k\to \infty} f_1^{k} (z_i^{'}) = \alpha_{i^*}$, where $i^* = i+1$ for even values of $i$, and
$i^* = i-1$ for odd values of $i$,
\item $\gamma_i$ joins $f_1^{p_i}(z_i^{'})$ and $z_i^{''}$, where $p_i>0$ and  $\lim _{k\to \infty} f_1^{k} (z_i^{''}) = \om_i$,
\item $\gamma _i$ and $\beta _i$ are disjoint from the $f_1$- orbits of every $z_j, z_j{'}, z_j^{''}$ with $i\neq j$.
\end{itemize}

By thickening the paths $\{\beta _i\}$ and $\{\gamma _i\}$, one can find free, pairwise dijsoint open disks
$\{D_i^{'}\}$ and $\{D_i^{''}\}$ such that the disks $D_i^{'}$ and $D_i^{''}$ are disjoint from the $f_1$-orbits
of the points $z_j, z_j^{'},$ and $z_j^{''}$, for $i\neq j$.

We construct a homeomorphism $h:\D\to \D$ such that:

\begin{itemize}
 \item $h = \id$ outside $\cup_{i\in \Z/4\Z} D_i^{'} \cup D_i^{''}$,
\item $h(z_i)= z_i^{'}$,
\item $h(f_1^p(z_i^{'})) = z_i^{''}$.
\end{itemize}

So, if we define $f = h\circ f_1 $, we obtain
$$\lim _{k\to \infty} f^{-k} (z_i) = \alpha_i, \ \lim _{k\to \infty} f^{k} (z_i) = \om _i,$$\noindent for all $i\in \Z/4\Z$.
Clearly we can make this construction in such a way that $f= f_1$ in a neighbourhood of $x$. Moreover, as the disks
$\{D_i^{'}\}$ and $\{D_i^{''}\}$ are free,
$$\fix(f) = \fix(f_1) = \{x\}.$$

So, $f$ realizes the elliptic cycle ${\cal L}$, but there is no simple closed curve of index $1$.

\subsubsection* {No negative-index fixed point is guaranteed by hyperbolicity.}

\begin{figure}[h]
\begin{center}
\psfrag{a}{$\alpha_0$}\psfrag{b}{$\alpha _3$}\psfrag{c}{$\om_1$}
\psfrag{d}{$\om_0$}\psfrag{e}{$\alpha_2$}\psfrag{f}{$\alpha_1$}\psfrag{g}{$\om_3$}
\psfrag{h}{$\om_2$}
\includegraphics[scale=0.6]{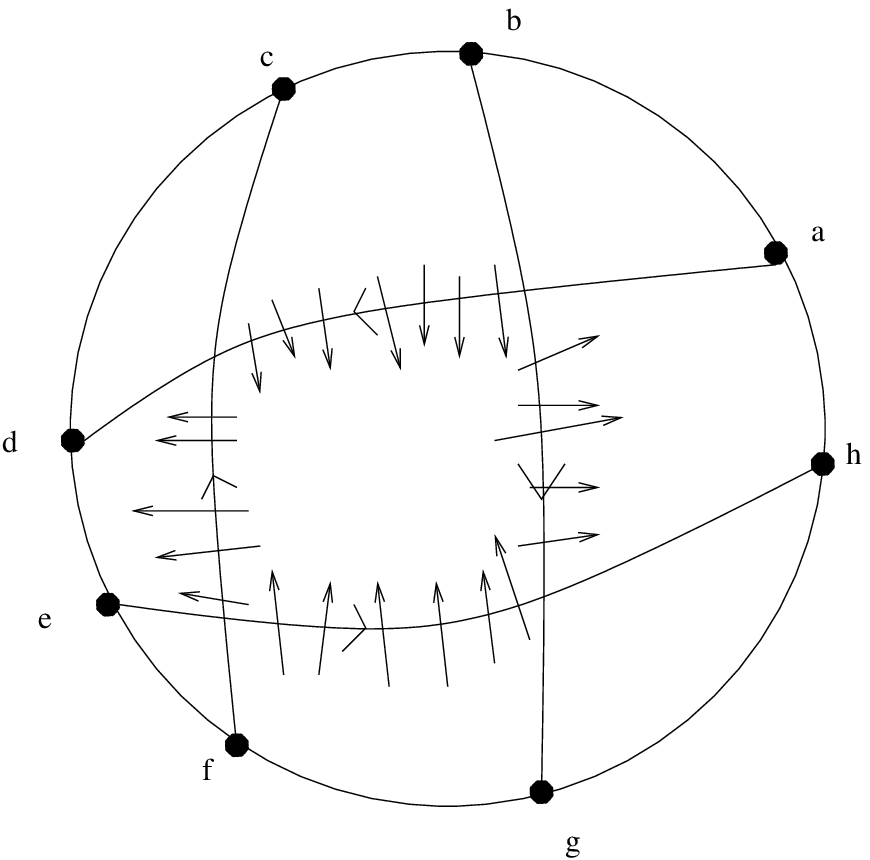}
\end{center}
\end{figure}

One could think that when ${\cal L}$ is hyperbolic, a negative-index fixed point should be obtained.  For example,
this would be the case if one had an oriented foliation ${\cal F}$ in $\D\backslash \fix (f)$ whose leaves are Brouwer lines
for $f$ and simple paths
$\gamma _i$, $i\in \Z/n\Z$ joining
$\alpha _i$ and $\om_i$ such that:
\begin{itemize}
 \item each $\gamma _i$ is positively transverse to ${\cal F}$,
\item the paths $\{\gamma _i\}$ bound  a compact disc in $\D$.
\end{itemize}
\noindent (See the figure above.)  Indeed, in this case, the Poincar\'e-Hopf formula would give a singularity $x$ of the foliation  for which
$i({\cal F}, x) < 0$.  So, $x\in \fix(f)$ and by a result of Le Calvez (\cite{patens})
one has $i(f,x) = i({\cal F}, x) < 0$.\\

However, this is not the case, as the following example shows.  Let $f_1$ be the time-one map of the flow whose orbits
are drawn in the figure below.

\begin{figure}[h]
\begin{center}
\psfrag{a}{$\alpha_0$}\psfrag{b}{$\alpha _3$}\psfrag{c}{$\om_1$}
\psfrag{d}{$\om_0$}\psfrag{e}{$\alpha_2$}\psfrag{f}{$\alpha_1$}\psfrag{g}{$\om_3$}
\psfrag{h}{$\om_2$}\psfrag{x}{$x$}
\includegraphics[scale=0.6]{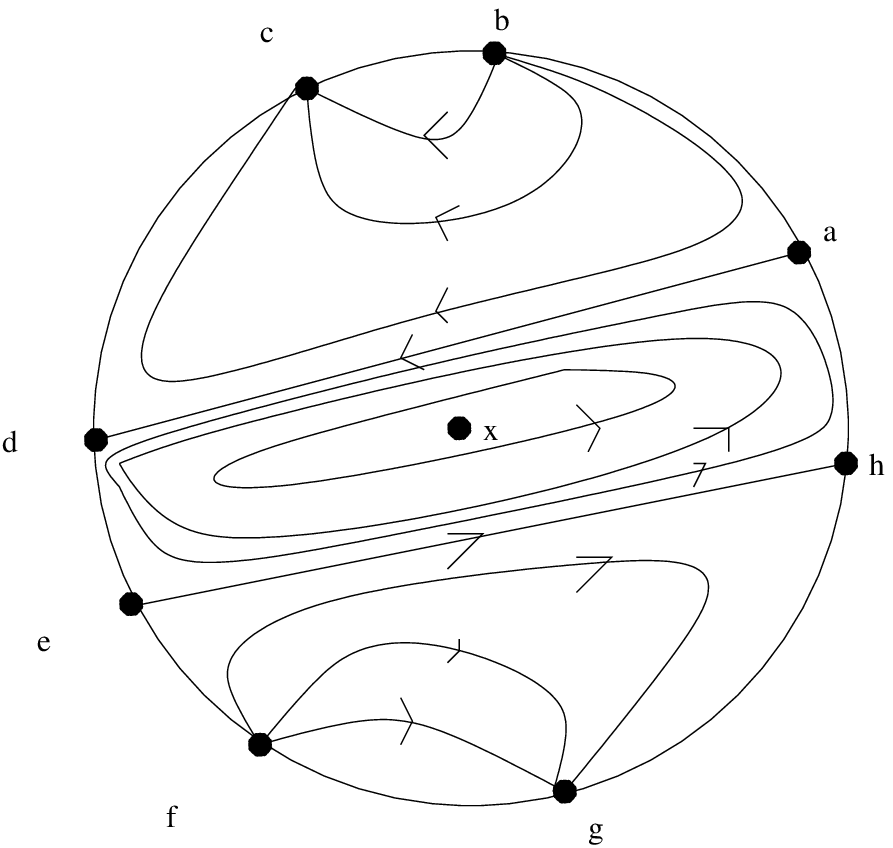}
\end{center}
\end{figure}

As we did in our preceding example, one can perturb $f_1$ in a homeomorphism $f$ such that:

\begin{itemize}
\item $\fix(f) = \fix(f_1) = \{x\}$,
\item $f = f_1$ in a neighbourhood of $x$,
\item $f$ realizes ${\cal L} = ((\alpha_i, \om_i))_{i\in\Z/4\Z}$.
\end{itemize}

So, $f$ realizes the hyperbolic cycle ${\cal L} $, but there is no fixed point of negative index.\\

The structure of this article is the following.  In Section \ref{bricks} we introduce the tools to be used (brick decompositions, Brouwer theory, Repeller/Attractor
configurations \cite{joul}) and we sum up the results from \cite{patrice} and \cite{joul} that will be used in the proofs.  In Section \ref{tech} we state two
lemmas that are the key for the contradiction argument in the proof of Theorem \ref{main*}, which is contained in Section \ref{proof}.  The last Section (\ref{last})
is devoted to the proof of Lemma \ref{opt}, which shows that out results are optimal.

\section{Preliminaries}\label{bricks}

\subsection{Brick decompositions}

A {\it brick decomposition}\index{brick! decomposition} $\cal D$ of an orientable surface $M$ is
a $1$- dimensional singular submanifold $\Sigma (\cal D)$ (the {\it
skeleton} of the decomposition), with the property that the set of singularities  $V$ is discrete and such that every 
$\sigma \in V$ has a neighborhood $U$
for which $U\cap (\Sigma ({\cal D})\backslash V)$ has exactly
three connected components. We have illustrated two brick decompositions
in Figure 4. The {\it bricks}\index{brick} are the closure of the
connected components of $M \backslash \Sigma (\cal D)$ and the {\it
edges} are the closure of the connected components of $\Sigma({\cal D})\backslash V$. We will write $E$ for the set of edges, $B$ for
the set of bricks and finally $\cal D$ = $(V,E,B)$ for a brick
decomposition.

\begin{figure}[h]\label{brickdec}
\begin{center}

   \subfigure[$M = \R ^2$]{\includegraphics[scale=0.4]{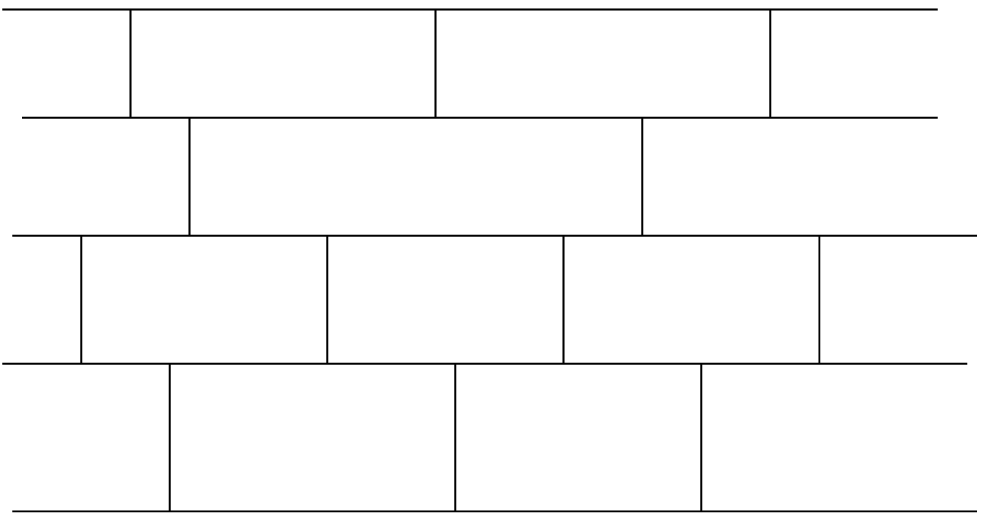}}\hspace{.25in}
    \subfigure[$M = \R ^2 \backslash \{0\}$]{\includegraphics[scale=0.5]{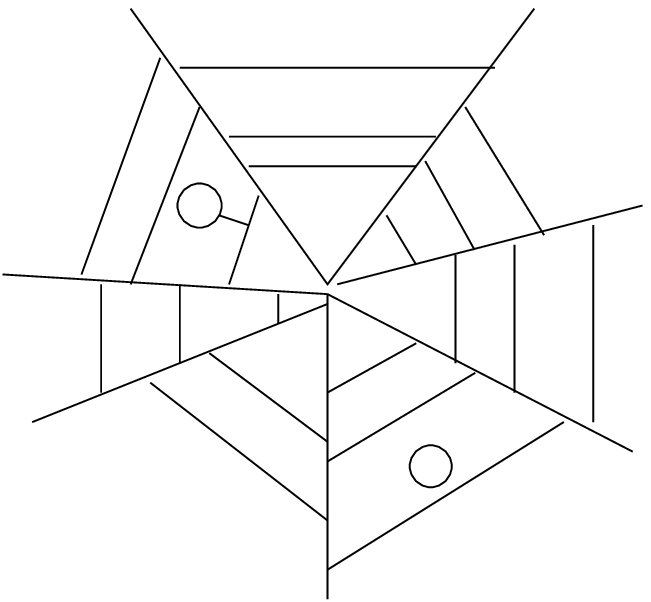}} \\
\caption{Brick decompositions}
\end{center}
 \end{figure}

Let $\cal D$ = $(V,E,B)$  be a brick decomposition of $M$.  We say
that $X\subset B$ is connected if given two bricks $b$, $b' \in X$,
there exists a sequence $(b_i)_{0\leq i \leq n}$, where $b_0 = b$,
$b_n = b'$ and such that $b_i$ and $b_{i+1}$ have non empty intersection, $i\in \{0, \ldots, n-1\}$.  Whenever
two bricks $b$ and $b'$ have no empty intersection, we say that they are {\it adjacent}\index{brick! adjacent}.  Moreover, we say that a brick
$b$ is {\it adjacent to a subset} $X\subset B$ if $b\notin X$, but $b$ is adjacent to one of the bricks in $X$.  We say that
$X\subset B$ is adjacent to $X'\subset B$ if $X$ and $X'$ have no common bricks but there exists $b\in X$ and $b'\in X'$
which are adjacent.

From now on we will identify a subset $X$ of $B$ with the closed subset of $M$ formed by
the union of the bricks in $X$.  By making so, there may be ambiguities (for instance, two adjacent subsets of $B$ have empty 
intersection in $B$ and nonempty intersection in $M$), but we will point it out when this happens. We remark that $\partial X$ is a
one-dimensional topological manifold and that the connectedness of
$X\subset B$ is equivalent to the connectedness of $X\subset M$ and
to the connectedness of $\inte (X) \subset M$ as well. We say that the decomposition $\cal D '$ is a 
{\it subdecomposition}\index{brick! subdecomposition}
of $\cal D$ if $\Sigma (\cal D ')$ $\subset \Sigma (\cal D)$.

If $f:M\to M$ is a homeomorphism, we define the application $\varphi
: \cal P (B) \to \cal P (B)$ as follows:
$$\varphi (X) = \{b\in B : f(X)\cap b
\neq \emptyset\}.$$\\

We remark that $\varphi (X)$ is connected whenever $X$ is.

We define analogously an application $\varphi _-: \cal P (B) \to \cal
P (B)$:

$$\varphi _- (X) =  \{b\in B :
f^{-1}(X)\cap b \neq \emptyset\}.$$\\

\begin{figure}[h]
\begin{center}
\psfrag{a}{$b$}\psfrag{b}{$f(b)$}\psfrag{c}{$\varphi (\{b\})$}

   {\includegraphics[scale=0.5]{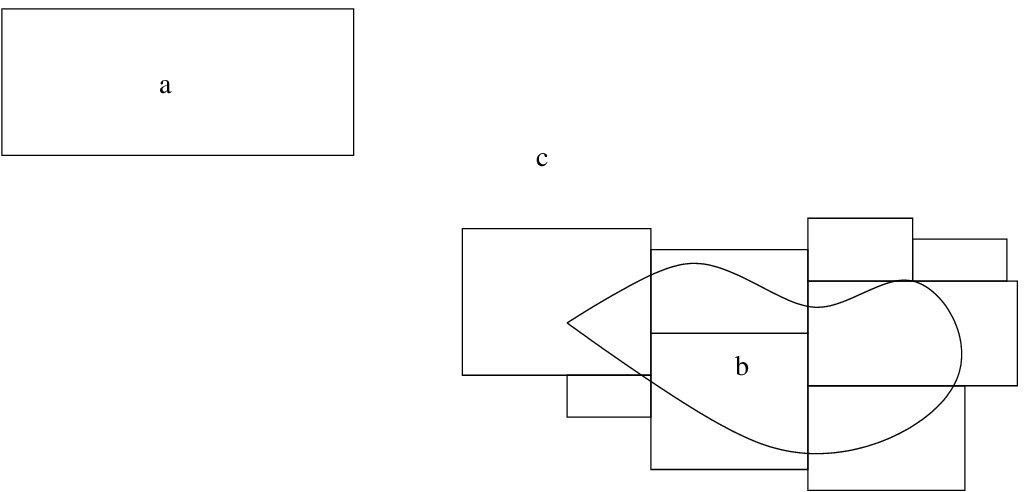}}
  \\

\end{center}
 \end{figure}

We define the {\it future}\index{brick! future} $[b]_\geq$ and the {\it past}\index{brick! past} $[b]_\leq$
of a brick $b$ as follows:

$$[b]_\geq = \bigcup _{k\geq 0} \varphi ^k (\{b\}),\  [b]_\leq =
\bigcup _{k\geq 0} \varphi _- ^k (\{b\}).$$

We also define the {\it strict future} $[b]_>$ and the {\it strict
past} $[b]_<$ of a brick $b$ :
$$[b]_> = \bigcup _{k>0} \varphi ^k (\{b\}),\  [b]_< = \bigcup
_{k> 0} \varphi _- ^k (\{b\}).$$

We say that a set $X\subset B$ is an {\it attractor}\index{attractor} if it verifies
$\varphi (X) \subset X$; this is equivalent in $M$ to the inclusion
$f(X)\subset \inte(X)$.  A {\it repeller}\index{repeller} is any set which verifies
$\varphi _- (X) \subset X$.  In this way, the future of any brick is an attractor,
and the past of any brick is a repeller. We observe that $X\subset B$
is a repeller if and only if $B\backslash X$ is an attractor.\\

\begin{obs}\label{br} The following properties can be deduced from the fact that $X\subset B$ is an attractor if and only if 
$f(X)\subset \inte (X)$:

\begin{enumerate}
 \item\label{br1} If $X\subset B$ is an attractor and $b\in X$, then $[b]_\geq \subset X$ ; 
if $X\subset B$ is a repeller and $b\in X$, then $[b]_\leq \subset X$,
\item\label{br2} if  $X\subset B$ is an attractor and $b\notin X$, then $[b]_\leq \cap X = \emptyset$ ;
if  $X\subset B$ is a repeller and $b\notin X$, then $[b]_\geq \cap X = \emptyset$,
\item\label{br3} if $b\in B$ is adjacent to the attractor $X\subset B$, then $[b]_> \cap X\neq \emptyset$;  
if $b\in B$ is adjacent to the repeller $X\subset B$, then $[b]_< \cap X\neq \emptyset$; 
\item\label{br4} two attractors are disjoint as subsets of $B$ if and only if they are disjoint as subsets of $M$;
in other words, two disjoint (in $B$) attractors  cannot be adjacent; respectively two disjoint (in $B$) repellers 
cannot be adjacent;
\end{enumerate}
 
\end{obs}

The following conditions are equivalent:
$$b\in [b]_>,\ [b]_> = [b]_\geq,\ b\in [b]_<,\ [b]_< = [b]_\leq,\ [b]_<\cap
[b]_\geq \neq \emptyset,\ [b]_\leq\cap [b]_> \neq \emptyset.$$

The existence of a brick $b\in B$ for which any of these conditions is 
satisfied is equivalent to the existence of a {\it closed chain of
bricks }, i.e a family  $(b_i)_{i\in \Z/r\Z}$ of bricks such that for
all $i\in \Z/r\Z$, $\cup _{k \geq 1} f^k(b_{i}) \cap b_{i+1}\neq
\emptyset$.  

In general, a {\it chain}\index{chain} for $f\in \homeo (M)$ is a family
$(X_i)_{0\leq i \leq r}$ of subsets of $M$ such that for all $0\leq i
\leq r-1$ ,
$\cup _{k\geq 1} f^k(X_i) \cap X_{i+1}\neq \emptyset$. We say that
the chain is
closed if $X_{r}= X_{0}$.

We  say that a subset $X\subset M$  is {\it
free}\index{free! set} if $f(X)\cap X= \emptyset$.

We  say that a brick decomposition $\cal D$ = $(V,E,B)$  is {\it
free}\index{free!brick decomposition} if every $b\in B$ is a free subset of $M$. If $f$ is
fixed point free it is always possible, taking sufficiently small
bricks, to construct a free brick decomposition.\\

We recall the definition of {\it maximal free decomposition}\index{maximal! free decomposition}, which
was introduced by Sauzet in his doctoral thesis \cite{sauzet}.  Let
$f$ be a fixed point free homeomorphism of a surface $M$. We say that
$\cal D$ is a maximal free decomposition if $\cal D$ is free and any
strict subdecomposition is no longer free.  Applying Zorn's lemma, it
is always possible to construct a maximal free 
subdecomposition of a given 
brick decomposition $\cal D$.

\subsection{Brouwer Theory background.}

We say that $\Gamma : [0,1]\to \overline \D$ is an {\it arc}\index{arc}, if it is continuous and injective.  We say that an arc $\Gamma$
joins $x\in \overline \D$ to $y\in \overline \D$, if $\Gamma (0)= x $ and $\Gamma (1)= y $.  We say that an arc $\Gamma$
joins $X\subset \overline \D$ to $Y\subset \overline \D$, if $\Gamma$ joins $x\in X$ to $y\in Y$.\\

Fix $f\in \homeo ^+ (\D)$.  An arc $\gamma$ joining $z\notin \fix (f)$ to $f(z)$ such that 
$f(\gamma)\cap \gamma = \{z, f(z)\}$ if $f^2(z) = z$ and $f(\gamma)\cap \gamma = \{ f(z)\}$ otherwise, is called a 
{\it translation arc}\index{arc! translation}.

\begin{prop}{\bf (Brouwer's translation lemma \cite{brouwer}, \cite{brown}, \cite{fathi} or \cite{guillou})}  If any
of the two following hypothesis is satisfyed, then there exists a  simple closed curve 
of index 1:

\begin{enumerate}
 \item  there exists  a translation arc $\gamma$ joining  $z\in \fix (f^2)\backslash \fix (f)$ to $f(z)$;
\item  there exists  a translation arc $\gamma$ joining $z\notin \fix (f^2)$ to 
$f(z)$ and an integer $k\geq 2$ such that  $f^k(\gamma) \cap \gamma \neq \emptyset$.
\end{enumerate}

\end{prop}

If $z\notin \fix (f)$, there exists a translation arc containing $z$; this is easy to prove once one has that the connected 
components of the complementary of $\fix (f)$ are invariant.  For a proof of this last fact, see \cite{brownkister} for a general proof in any dimension, or \cite{duke}
for an easy proof in dimension 2.

We deduce:

\begin{cor} If $\per (f) \backslash \fix (f) \neq \emptyset$, then  there exists a  simple closed curve 
of index 1. 
 
\end{cor}

\begin{prop} {\bf (Franks' lemma \cite{pbf})} If there
exists a closed chain of free, open and pairwise disjoint disks for $f$, then there exists a simple closed curve of index 1.
 
\end{prop}

Following Le Calvez \cite{patrice}, we will say that $f$ is {\it recurrent}\index{recurrent homeomorphism} if there
exists a closed chain of free, open and pairwise disjoint disks for $f$.\\

The following proposition is a refinement of Franks' lemma due to Guillou and  Le Roux (see \cite{leroux}, page 39).

\begin{prop}\label{guile} Suppose there exists a closed chain
$(X_i)_{i\in\Z/r\Z}$ for $f$ of free subsets  whose
interiors are pairwise disjoint and which verify the following
property: given any two points $z,z'\in X_i$ there exists an
arc $\gamma$ joining $z$ and $z'$ such that $\gamma\backslash\{z,z'\}\subset
\inte (X_i)$. Then, $f$ is recurrent.
 
\end{prop}

We deduce:
\begin{prop}\label{franksfino} Let $\cal D$ = $(V,E,B)$  be a free
brick decomposition of $\D \backslash \fix (f)$.  If there exists
$b\in B$ such that $b\in [b]_>$, then $f$ is recurrent.
 
\end{prop}

\subsection{Little bricks at infinity.}

Fix  $f\in \homeo ^+ (\D)$, different from the identity map and  {\it non-recurrent}. We will make use of the following two propositions from
\cite{patrice} (both of them depend on the non-recurrent character of $f$).  The first one (Proposition 2.2 in \cite{patrice})
 is a refinement of a result already appearing in \cite{sauzet}; the second one is Proposition 3.1 in \cite{patrice}.

\begin{prop}[\cite{sauzet},\cite{patrice}]\label{futcon}  Let $\cal D$ = $(V,E,B)$  be a
free maximal brick decomposition of $\D\backslash \fix (f)$.  Then,
 the sets $[b]_\geq$, $[b]_>$, $[b]_\leq$ and $[b]_<$ are connected.  In particular every connected component of an attractor
is an attractor, and  every connected component of a repeller is a repeller.

\end{prop}

\begin{prop}\cite{patrice}\label{gamak} If $f$  satisfies the
hypothesis of Theorem \ref{main*}, then for all $i\in \Z/n\Z$ we can find a sequence of arcs
$(\gamma _i^k)_{k\in\Z}$ such that:
\begin{enumerate}\item[$\bullet$] each $\gamma _i^k$ is a translation
arc from $f^k(z_i)$ to $f^{k+1}(z_i)$,
\item[$\bullet$] $f(\gamma _i^k)\cap \gamma _i^{k'} = \emptyset$ if
$k'<k$,
\item[$\bullet$] the sequence $(\gamma _i^k)_{k\leq 0}$ converges to
$\{\alpha _i\}$ in the Hausdorff topology,
\item[$\bullet$] the sequence $(\gamma _i^k)_{k\geq 0}$ converges to
$\{\om _i\}$ in the Hausdorff topology.
\end{enumerate}

\end{prop}

This result is  a consequence of
Brouwer's translation lemma and the hypothesis on the orbits of the
points $(z_i)_{i\in\Z/n\Z}$. In particular,   the
extension hypothesis of Theorem \ref{main*} is used. It
allows us to construct a particular
brick decomposition suitable for our purposes: \\

\begin{lema}\label{ends} For every $i\in \Z/n\Z$, take $U_i^-$ a neighbourhood of $\alpha _i$ in $\overline \D$ and 
$U_i^+$ a neighbourhood of $\om _i$ in $\overline \D$ such that $U_i^-\cap U_i^+ = \emptyset$.
There exists two families $(b_i'^l)_{i\in\Z/n\Z,l\geq 1}$ and
$(b_i'^l)_{i\in\Z/n\Z,l\leq -1}$ of closed disks in $\D$, and a family of integers $(l_i)_{i\in \Z/n\Z}$ such that:
\begin{enumerate}\item each $b_i'^l$ is free and contained in $U_i^-$
($l\leq-1$) or in $U_i^+$ ($l\geq 1$),
\item $\inte (b_i'^l) \cap \inte (b_i'^{l'}) = \emptyset$, if $l \neq l'$
,
\item for every $k>1$ the sets $(b_i'^l)_{1\leq l \leq k}$ and
$(b_i'^l)_{-k \leq l \leq -1 }$ are connected, 

\item for all $i\in\Z/n\Z$, $\partial \cup_{l\in \Z\backslash \{0\}} b_i'^l$ is a one dimensional submanifold,
\item if $x\in \D$, then $x$ belongs to at most two different disks in the family $(b_i'^l)_{l\in \Z\backslash \{0\}}$,
$i\in\Z/n\Z$,
\item\label{pto1} for all  $i\in \Z/n\Z$  $f^{l_i + l} (z_i) \in \inte(b_i'^{l+1}) $ for all $l\geq 0$,
and  $f^{-l_i-l} (z_i) \in
\inte(b_i'^{-l-1}) $ for all $l\geq 0$,
\item $f^k(z_j) \in b_i'^l$ if and only if $j=i$ and $k = l_i+l-1$,
\item the sequence $(b_i'^l)_{l \geq 1}$ converges to  $\{\om _i\}$ in the Hausdorff topology
and the sequence $(b_i'^l)_{l \leq -1}$ converges to $\{\alpha _i\}$ in the Hausdorff topology.

\end{enumerate}

\end{lema}

The idea is to construct  trees $T_i^-\subset U_i^-, T_i^+ \subset U_i^+$, $i\in \Z/n\Z$ by deleting the loops of the curves 
$\prod _{k\geq -1} \gamma _i^k\cap U_i^-$ and 
 $\prod _{k\leq 1} \gamma _i^k\cap U_i^+$ respectively, and then
thickening these trees to obtain the families $(b_i'^l)_{i\in\Z/n\Z,l\geq 1}$
and $(b_i'^l)_{i\in\Z/n\Z,l\leq -1}$. We refer the reader to \cite{joul}
for a proof in english but we remark that these results are contained in \cite{patrice}. We have illustrated these families in Figure 4.

\begin{figure}[h]\label{fliabil}
\begin{center}
\psfrag{a}{$b_3'^{-l}$}\psfrag{b}{$b_1'^l$}\psfrag{c}{$b_0'^l$}\psfrag{d}{$b_2'^{-l}$}\psfrag{e}{$b_1'^{-l}$}\psfrag{f}{$b_3'^l$}
\psfrag{g}{$b_2'^{l}$}
\psfrag{h}{$b_0'^{-l}$}\psfrag{i}{$\alpha_3$}\psfrag{j}{$\om_1$}\psfrag{k}{$\om_0$}\psfrag{l}{$\alpha_2$}\psfrag{m}{$\alpha_1$}
\psfrag{n}{$\om_3$}\psfrag{o}{$\om_2$}\psfrag{p}{$\alpha_0$}
\includegraphics[scale=0.6]{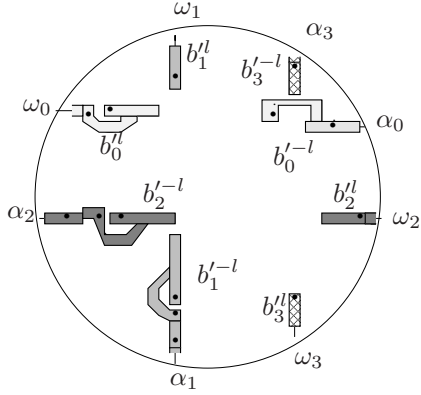}
\caption{The families $b_i'^l$}
\end{center}
\end{figure}

\begin{obs}\label{puntas}  The fact that the sequence $(b_i'^l)_{l\geq 1}$ converges in the 
Hausdorff topology to $\om _i$, implies that we can find an 
arc  $\Gamma_i ^+ : [0,1]\to \inte (\cup _{l\geq 0 } 
b_i'^{l}) \cup \{\om _i\}$ such that  $\Gamma_i ^+(1) = \om _i$, $i\in \Z/n\Z$.  Similarly, we can find an arc
$\Gamma_i ^- : [0,1]\to \inte (\cup _{l\geq 0 } b_i'^{-l}) \cup \{\alpha _i\}$ such that $\Gamma_i ^-(1) = \alpha _i$, 
$i\in\Z/n\Z$.
\end{obs}

\subsection{Repeller/ Attractor configurations}\label{raconf}

\subsubsection{Cyclic order at infinity.}\label{co}

Let $(a_i)_{i\in \Z/n\Z}$ be a family of  non-empty, pairwise disjoint, closed, connected subsets of 
$\D$, such that $\overline a_i \cap \partial \D \neq \emptyset$ and $U= \D\backslash (\cup_{i\in \Z/n\Z} a_i)$ is
a connected open set.  As $U$ is connected, and its complementary set in $\C$
$$\{z\in \C: |z|\geq 1\}\cup \cup_{i\in \Z/n\Z} a_i$$\noindent is connected, $U$ is simply connected.   

With these hypotheses, there is a natural cyclic order on the sets $\{a_i\}$.  Indeed, $U$ is conformally isomorphic
to the unit disc via the Riemann map $\varphi : U \to \D$, and one can consider the Carath\'eodory's extension of $\varphi$,

$$\hat\varphi: \hat U \to \overline{\D},$$\noindent which is a homeomorphism between the prime ends completion $\hat U$
of $U$ and the closed unit disk $\overline {\D}$.  The set $\hat J_i$ of prime ends whose impression is contained in $a_i$ is open and connected.  It follows that
the images $J_i = \hat\varphi (\hat J_i)$ are pairwise disjoint open intervals in $S^1$, and are therefore cyclically ordered
following the positive orientation in the circle. 

\subsubsection{Repeller/Attractor configurations. }

We recall de definition of Repeller/Attractor configuration that was introduced in \cite{joul}.

We fix $f\in \homeo ^+ (\D)$ together with a free maximal decomposition in bricks  $\cal
D$$=(V,E,B)$ of $\D\backslash \fix (f)$ .\\

Let  $(R_i)_{i\in \Z/n\Z}$ and  $(A_i)_{i\in \Z/n\Z}$ be two families of connected, pairwise disjoint subsets
 of $B$ such that  :

\begin{enumerate}
 
\item For all $i\in \Z/n\Z$:

\begin{enumerate} \item  $R_i$ is a repeller and $A_i$ is an attractor;

\item there exists non-empty, closed, connected subsets of $\D$, $r_i\subset \inte(R_i)$, $a_i\subset
\inte(A_i)$
such that $\overline{r_i} \cap \partial \D\neq \emptyset$ and $\overline{a_i} \cap \partial \D\neq \emptyset$ ,
\end{enumerate}

\item \label{orden} $\D\backslash (\cup_{i\in \Z/n\Z}(a_i\cup r_i))$ is a connected open set.
\end{enumerate}

We say that the pair $((R_i)_{i\in \Z/n\Z} , (A_i)_{i\in \Z/n\Z})$ is a {\it Repeller/Attractor configuration of order $n$}
\index{Repeller/Attractor configuration}.\\ 
We will note 
$${\cal{E}} = \{R_i, A_i: i\in \Z/n\Z \}.$$

Property \ref{orden} in the previous definition allows us to  give a cyclic order to the sets $r_i, a_i, i\in \Z/n\Z$ (see
the beginning of this section).

We say that a Repeller/Attractor configuration of order $n\geq 3$ is an {\it elliptic configuration} if :

\begin{enumerate}
\item the cyclic order of the sets $r_i, a_i$, $i\in \Z/n\Z$, satisfies the 
{\it elliptic order property}\index{elliptic order property}:
 $$a_0\to r_2\to a_1\to \ldots\to a_i\to r_{i+2}\to a_{i+1}\to \ldots\to a_{n-1}\to r_1\to a_0.$$
\item for all $i\in\Z/n\Z$ there exists a brick $b_i\in R_i$ such that $[b_i]_{\geq}\cap A_i\neq \emptyset $;

\end{enumerate}

 We say that a Repeller/Attractor configuration is a {\it hyperbolic configuration} if:
\begin{enumerate}
\item  the cyclic order of the sets $r_i, a_i$, $i\in \Z/n\Z$, satisfies the
 {\it hyperbolic order property}\index{hyperbolic order property}:
$$r_0\to a_0\to r_{1}\to a_1\to \ldots\to r_i\to a_{i}\to r_{i+1}\to a_{i+1}\to\ldots \to r_{n-1}\to 
 a_{n-1}\to r_0.$$ 
 \item for all $i\in\Z/n\Z$ there exists two bricks $b_i^i, b_{i}^{i-1}\in R_i$ such that $[b_i^i]_>\cap A_i\neq \emptyset$, and 
$[b_i^{i-1}]_>\cap A_{i-1}\neq \emptyset$;

\end{enumerate}

\begin{figure}[h]   \label{noorientado}
\begin{center}
\psfrag{a}{$R_0$}\psfrag{b}{$A_0$}\psfrag{c}{$R_1$}\psfrag{d}{$A_1$}\psfrag{e}{$R_2$}\psfrag{f}{$A_2$}\psfrag{m}{$R_3$}
\psfrag{n}{$A_3$}\psfrag{g}{$R_1$}\psfrag{h}{$A_0$}\psfrag{i}{$R_2$}\psfrag{j}{$A_1$}\psfrag{k}{$R_0$}\psfrag{l}{$A_2$}

   \subfigure[An elliptic configuration]{\includegraphics[scale=0.3]{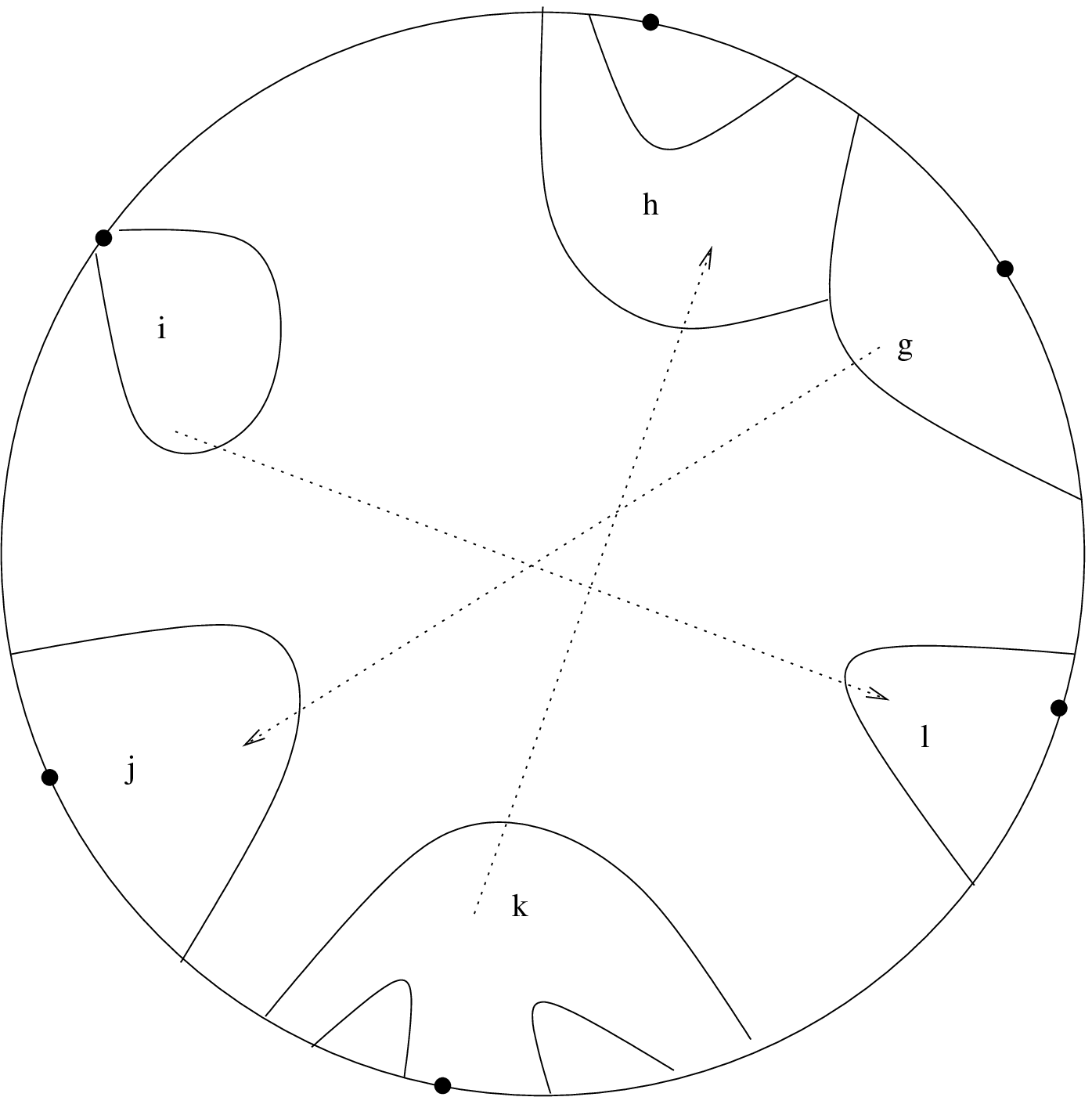}}\hspace{.25in}
    \subfigure[ A hyperbolic configuration] {\includegraphics[scale=0.3]{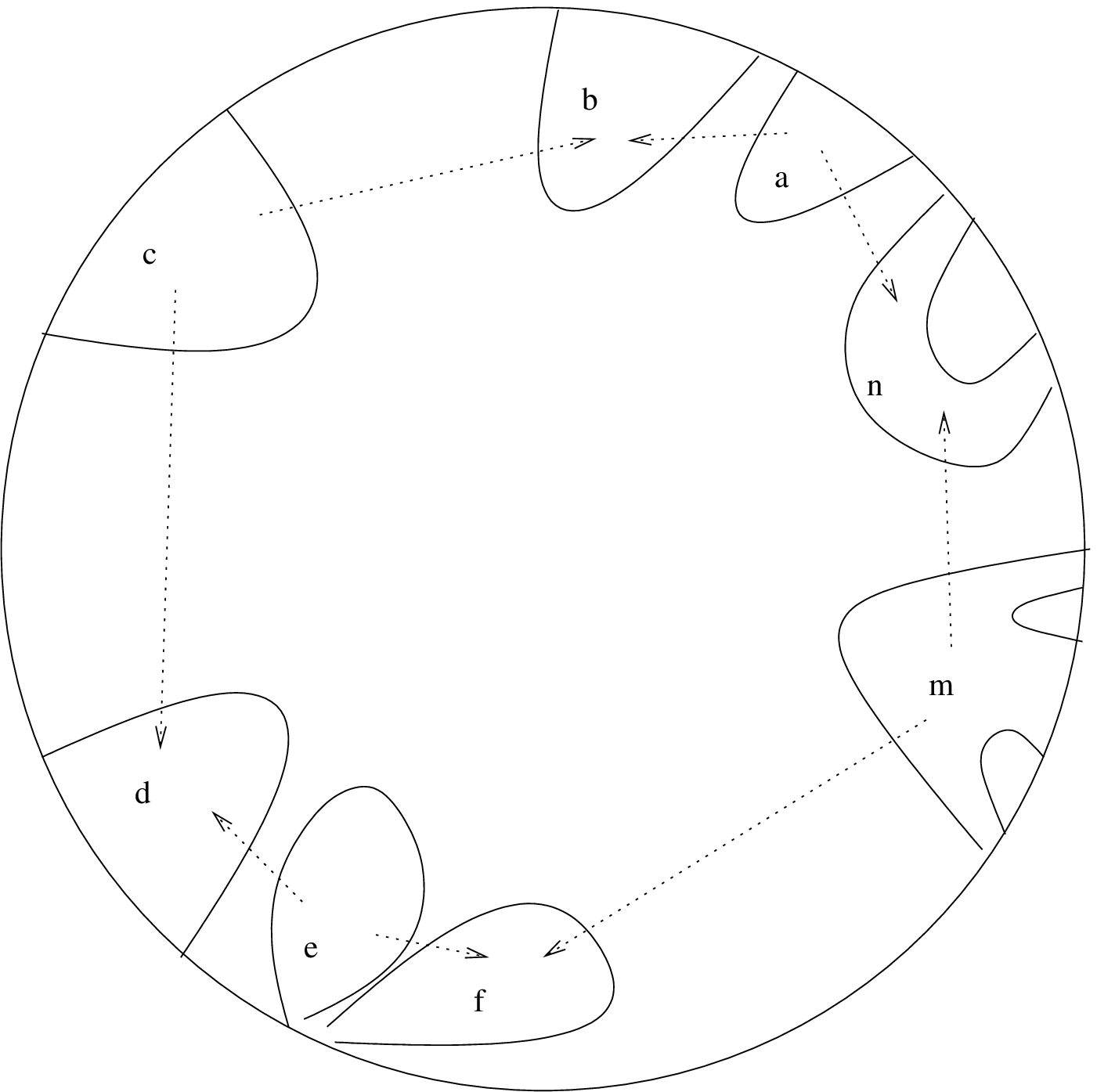}}\hspace{.25in} \\

\end{center}
 \end{figure}

We will make use of the following results from \cite{joul}:

\begin{prop}{\bf \cite{joul}}\label{hc} If there exists an elliptic configuration of order $n\geq 3$, then $f$ is recurrent.
 
\end{prop}

\begin{prop}{\bf \cite{joul}}\label{2c} If there exists a hyperbolic configuration of order $n\geq 2$, then $\fix (f)\neq \emptyset$.
 
\end{prop}

\section{Two technical lemmas.}\label{tech}

In this section we give applications of Propositions \ref{hc} and \ref{2c} respectively, that will be used in the proof
of Theorem \ref{main*}.

We fix $f\in \homeo ^+ (\D)$ together with a free maximal decomposition in bricks  ${\cal D} =(V,E,B)$
of $\D\backslash \fix (f)$, and we suppose that $f$ is non-recurrent.

Let $a_i$, $i\in \Z/n\Z$, be non-empty, pairwise disjoint, closed, connected subsets of
$\D$, such that $\overline a_i \cap \partial \D \neq \emptyset$, for all $i\in \Z/n\Z$, and $U=
\D\backslash (\cup_{i\in \Z/n\Z} a_i)$ is a connected open set. We consider the  Riemann map
$\varphi: U \to \D$, and the open
intervals on the circle $J_i, i\in\Z/n\Z$  defined in 3.1. We recall that the interval $J_i$ correspond to the
prime ends in $U$ whose impression is contained in $a_i$.

Let $(I_i)_{i\in\Z/n\Z}$ be the connected components of $S^1\backslash (\cup_{i\in\Z/n\Z} J_i)$.  So, each $I_i$ is
a closed interval, that may be reduced to a point.

\begin{obs}\label{ii} One can cyclically order the sets  $(a_i)_{i\in \Z/n\Z}$,
 $(r_j)_{i\in \Z/m\Z}$, where  $(r_j)_{i\in \Z/m\Z}$ is any family of closed, connected and pairwise disjoint subsets of
$U$ satisfying:

\begin{enumerate}
 \item $\overline{r_j}\cap \partial U \neq \emptyset$, $j\in \Z/m\Z$,
\item for all $j\in \Z/m\Z$, there exists $i_j\in \Z/n\Z$ such that $\overline{\varphi (r_j)}\cap S^1 \subset I_{i_j}$,
\item the correspondence $j\to i_j$ is injective.
\end{enumerate}

\end{obs}

\begin{lema}\label{quid}

We suppose that:

\begin{enumerate}

\item the cyclic order of the sets $a_i$, $i\in \Z/n\Z$, is the following:
$$a_0\to a_1\to \ldots \to a_{i}\to a _{i+1}\to \ldots \to a_{n-1}\to a_0.$$

\item  for all $i\in \Z/n\Z$ there exists $b_i^+ \in B $, such that  $a_i \subset [b_i^+]_>$,
\item there exists three bricks $(b_s^-)_{s\in \Z/3\Z}$ such that

\begin{enumerate} \item for all $s\in \Z/3\Z$ and for all $i\in\Z/n\Z$, one has $b_s^-\subset[b_i^+]_<$
(and so $[b_s^-]_<\subset U $),
\item  $\overline{[b_s^-]_<}\cap \partial U \neq \emptyset$ for all $s\in\Z/3\Z$,

\item for all $s\in\Z/3\Z$ there exists $i_s\in \Z/n\Z$ such that $\overline{\varphi ([b_s^-]_<)}\cap S^1\subset I_{i_s}$,

\end{enumerate}

\end{enumerate}

Then, the correspondence $s\to i_s$ is not injective.

\end{lema}

\begin{figure}[h]\label{fquid}
\begin{center}
\psfrag{a}{$[b_0^-]_<$}\psfrag{b}{$[b_1^-]_<$}\psfrag{c}{ $[b_2^-]_<$}
\psfrag{f}{$a_{j_0}$}
\psfrag{e}{ $a_{j_2}$}
\psfrag{g}{ $A_2$}\psfrag{h}{ $A_0$}\psfrag{i}{ $A_1$}
\psfrag{l}{ $a_{j_1}$}\psfrag{x}{ $R_0$}\psfrag{y}{ $R_1$}\psfrag{z}{ $R_2$}
\includegraphics[scale=0.6]{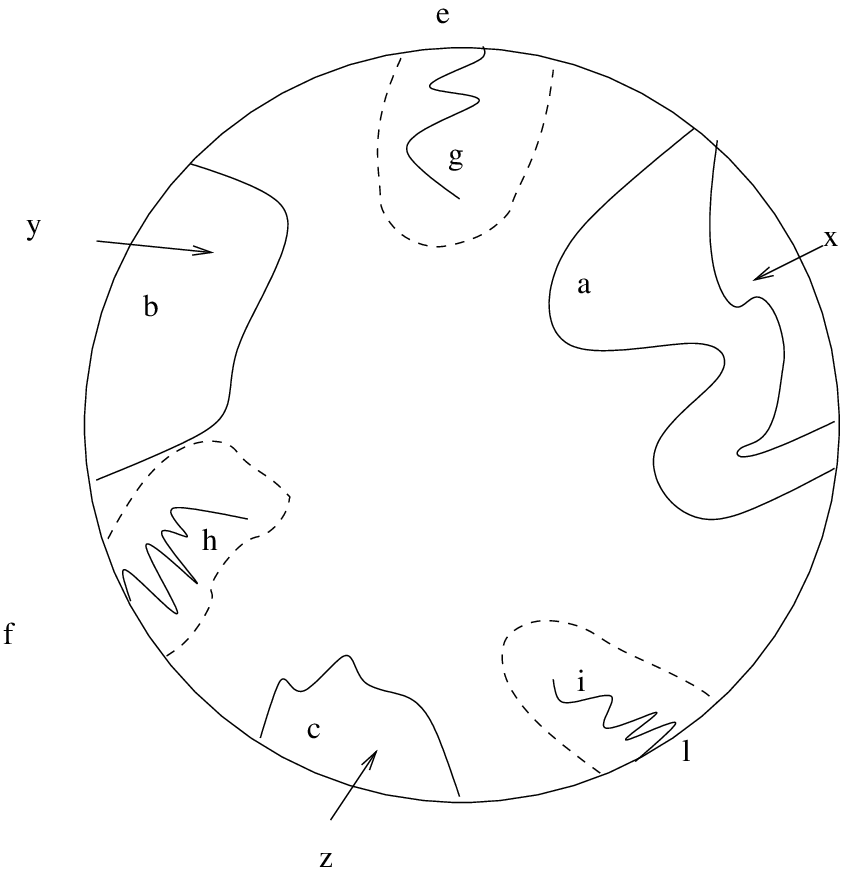}
\caption{Lemma \ref{quid}}
\end{center}
\end{figure}

\begin{proof} We will prove that if the correspondence $s\to i_s$ is injective, we can construct an elliptic configuration of
order $3$. As we are assuming $f$ is not recurrent, this is not possible by Proposition \ref{hc}.

We begin by proving that $[b_s^-]_<\cap [b_r^-]_<\neq \emptyset$ implies $i_s = i_r$.  Indeed,  if
$[b_s^-]_<\cap [b_r^-]_<\neq \emptyset$, then
 $[b_s^-]_<\cup [b_r^-]_<$ is a connected set and $\overline{\varphi([b_s^-]_<\cup [b_r^-]_<)}$ intersects both
$I_{i_s}$ and $I_{i_r}$.  If $i_s\neq i_r$, then there exists $j_0, j_1\in \Z/n\Z$ such that any arc joining $J_{j_0}$ and
$J_{j_1}$ separates $I_{i_r}$ from $I_{i_s}$ in $\overline \D$ .  Our hypothesis 3.(a) allows us to take a crosscut
$\gamma$ from $a_{j_0}$ to $a_{j_1}$ such that $\gamma \cap U \subset[b_s^-]_>$.  So, $\overline{\varphi(\gamma \cap U)}$
is an arc joining $J_{j_0}$ and $J_{j_1}$, and
$$\overline{\varphi(\gamma \cap U)} \cap \varphi([b_s^-]_<\cup [b_r^-]_<)
\neq \emptyset. $$\noindent  This gives us $$([b_s^-]_<\cup [b_r^-]_<)\cap[b_s^-]_>\neq \emptyset,$$ and as we are supposing
that $f$ is not recurrent, $$ [b_r^-]_<\cap[b_s^-]_>\neq \emptyset.$$\noindent  So, $$ [b_s^-]_<\subset [b_r^-]_<, $$
\noindent which implies
$$\overline{\varphi ([b_s^-]_<)}\cap S^1\subset I_{i_s}\cap I_{i_r}, $$ \noindent a contradiction.

So, if the correspondence $s\to i_s$ is injective, the sets $[b_s^-]_<$ are pairwise disjoint, and one can cyclically order the
$n+3$  sets $a_i, [b_s^-]_<$, $i\in \Z/n\Z$,
$s\in \Z/3\Z$ (see Remark \ref{ii}).  We may suppose without loss of generality that

$$[b_0^-]_<\to [b_1^-]_< \to [b_2^-]_< \to [b_0^-]_< .$$

\noindent For all $s\in \Z/3\Z$, we can take $j_s\in \Z/3\Z$ such that

 $$[b_0]_<^{-}\to a_{j_2} \to [b_1^-]_< \to a_{j_0}\to [b_2^-]_< \to a_{j_1}\to [b_0]_<^{-}$$ \noindent (see Figure 9).

For all $s\in \Z/3\Z$, we define: $$R_s = [b_s^-]_< , \ A_s = [b_{j_s}^+]_> .$$  We want to show that
$$((R_s)_{s\in \Z/3\Z}), (A_s)_{s\in \Z/3\Z}),$$ \noindent is an elliptic configuration.
 It is enough to show that the sets $A_s, R_s$, $s\in \Z/3\Z$, are
pairwise disjoint, because of the cyclic order of these sets , and our hypothesis 3.(a).
  We already know that the sets $R_s$, $s\in \Z/3\Z$, are pairwise disjoint. As we are
supposing that $f$ is not recurrent, and
$b_{j_s}^+\in [b_{s'}^-]_>$ for every pair of indices $s, s'$ in $\Z/3\Z$ (3.(a)), we know that
$$[b_{j_s}^+]_>\cap [b_{s'}^-]_<=\emptyset$$\noindent for all
$s, s'$ in $\Z/3\Z$.  So, the sets $\{A_s\}$, are disjoint from the sets $\{R_s\}$, and we just have to show
that the sets $\{A_s\}$ are pairwise disjoint to finish the proof of the lemma.

Because of the symmetry of the
problem it is enough to show that $$A_0\cap A_1=\emptyset . $$\noindent  If this is not so,
$$A_0\cup A_1 = [b_{j_0}^+]_> \cup [b_{j_1}^+]_> $$ \noindent would be a connected set containing  both $a_{j_1}$ and
$a_{j_0}$, and the cyclic order would imply that
$$( [b_{j_0}^+]_> \cup [b_{j_1}^+]_>)\cap [b_{j_0}^+]_<\neq \emptyset , $$\noindent  by our hypothesis 3.(a).  As we
are supposing that $f$ is not recurrent, we have
$$ [b_{j_1}^+]_>\cap [b_{j_0}^+]_<\neq \emptyset . $$\noindent But this implies that $[b_{j_1}^+]_> $ is a connected set
containing both $a_{j_1}$ and $a_{j_0}$.  Once again  our hypothesis 3.(a) and the cyclic order gives us
$$[b_{j_1}^+]_> \cap [b_{j_1}^+ ]_<\neq \emptyset , $$\noindent and we are done.

\end{proof}

For our next lemma, we keep the assumption on the cyclic order of the sets $a_i, i\in \Z/n\Z$:

$$a_0\to a_1\to \ldots \to a_{i}\to a _{i+1}\to \ldots \to a_{n-1}\to a_0.$$

\noindent We define $I_i$, as to be the connected component of $S^1\backslash \cup_{j\in \Z/n\Z} J_j$  that
follows $J_{i-1}$ in the natural cyclic order on $S^1$, so that we have:
$$J_{i-1}\to I_i\to J _{i},$$

\noindent for all $ i\in \Z/n\Z$.

\begin{lema}\label{quidh} If for all $ i\in \Z/n\Z$:

\begin{enumerate}

\item there exists $b_i^+ \in B $, such that  $a_i \subset [b_i^+]_>$,
\item there exists $b_i^- \in B $ such that $b_{i}^-\subset[b_j^+]_<$, $j\in \{i-1,i\}$,
\item $[b_i^-]_<\subset U$, and   $\overline{[b_i^-]_<}\cap \partial U\neq \emptyset$,
\item $\overline{\varphi([b_i^-]_<)}\cap S^1 \subset I_{i}$,

\end{enumerate}
then $\fix(f)\neq \emptyset$.

\end{lema}

\begin{figure}[h]\label{fquidh}
\begin{center}
\psfrag{a}{ $[b_0^-]_<$}\psfrag{d}{ $[b_1^-]_<$}
\psfrag{b}{  $[b_2^-]_<$}\psfrag{e}{ $[b_3^-]_<$}
\psfrag{c}{ $[b_4^-]_<$}\psfrag{f}{ $[b_5^-]_<$}
\psfrag{g}{ $a_{4}$}\psfrag{h}{ $a_{5}$}\psfrag{i}{ $a_{0}$}\psfrag{j}{ $a_{1}$}\psfrag{k}{ $a_{2}$}
\psfrag{l}{ $a_{3}$}
\includegraphics[scale=0.6]{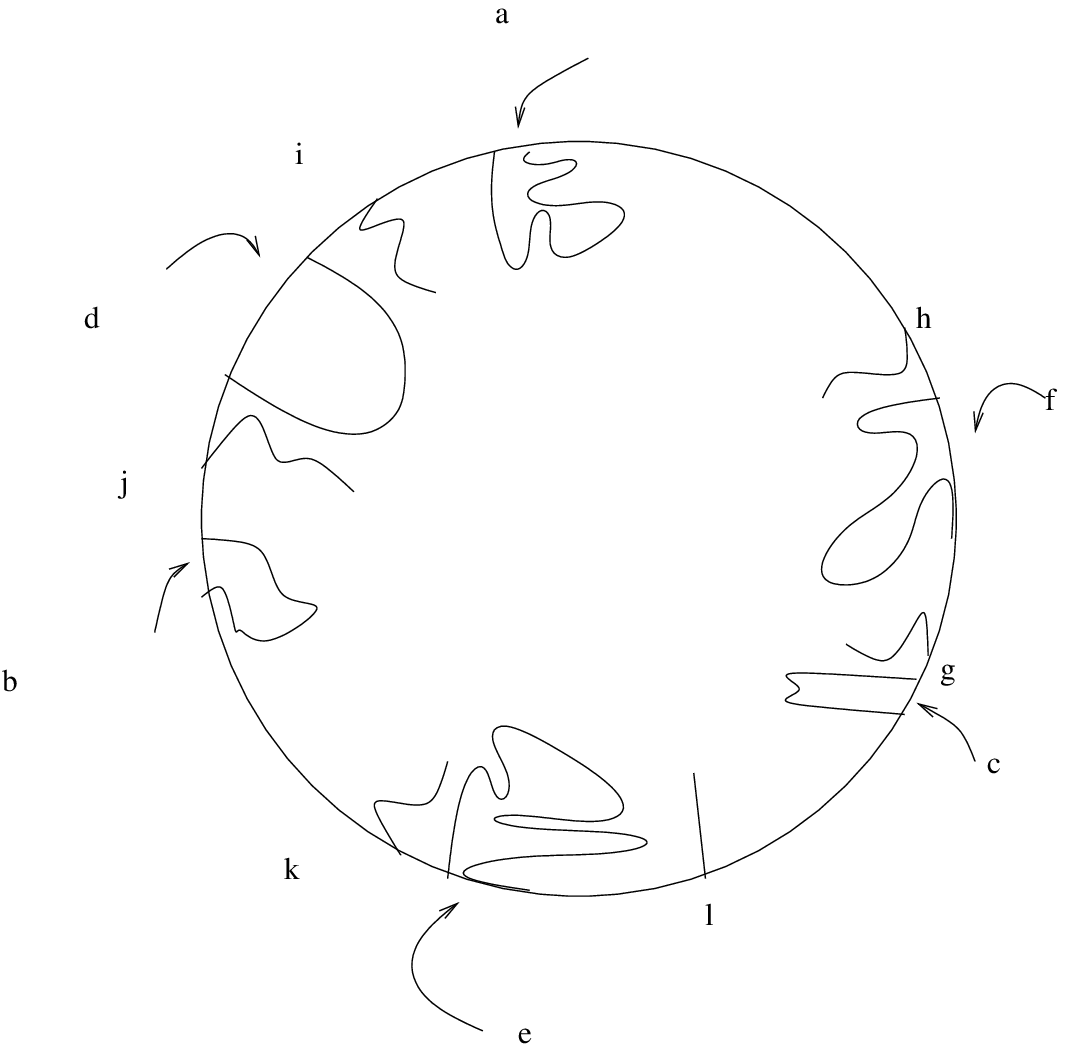}
\caption{Lemma \ref{quidh} with $n=6$}
\end{center}
\end{figure}

\begin{proof} By Proposition \ref{2c} it is enough to show that we can construct
a hyperbolic configuration.

We begin by proving that the sets $\{[b_i^-]_<\}$,
are pairwise disjoint.  Otherwise, there exists $i\neq j$, such
that $$[b_i^-]_<\cap [b_j^-]_<\neq \emptyset.$$ Then, $[b_i^-]_<\cup [b_j^-]_<$ is a connected set and
$\overline{\varphi([b_i^-]_<\cup [b_j^-]_<)}$ intersects both $I_{i}$ and $I_{j}$.  The cyclic order implies that any arc
joining
 $J_{i-1}$ and $J_{i}$ separates $I_i$ from $I_j$, $i\neq j$.

Our hypothesis 2. allows us to take
a crosscut $\gamma $ from $a_{i-1}$ to $a_{i}$ such that $$\gamma \cap U \subset [b_i^-]_>.$$ So,
$\overline {\varphi (\gamma \cap U )}$ is an arc joining $J_{i-1}$ and $J_{i}$, and $$\overline {\varphi (\gamma \cap U )}\cap
\varphi([b_i^-]_<\cup [b_j^-]_<)\neq \emptyset.$$  This gives us
$$([b_i^-]_<\cup [b_j^-]_<)\cap  [b_i^-]_>\neq \emptyset,$$ and as we are supposing that $f$ is not recurrent,
$$[b_j^-]_<\cap  [b_i^-]_>\neq \emptyset.$$  So, $ [b_i^-]_<\subset [b_j^-]_< $, which implies
$$\overline{\varphi([b_i^-]_<)}\cap S^1\subset I_{i}\cap I_{j},$$ a contradiction.

So, we can cyclically order the $2n$ sets $a_i$, $[b_i^-]_<$, $i\in \Z/n\Z$ (see Remark \ref{ii}).
Moreover, for all $i\in \Z/n\Z$,
$$a_{i-1}\to [b_i^-]_<\to  a_i.$$

Define $A_i = [b_i^+]_>$ and $R_i=  [b_i^-]_<$, for $i\in \Z/n\Z$.  To finish the proof
of the lemma, it is enough to show that the sets $R_i, A_i, i\in \Z/n\Z,$ are pairwise disjoint.  Indeed, if this
is true, our previous remark on the cyclic order, and our hypothesis 2. imply that $((R_i)_{i\in \Z/n\Z},(A_i)_{i\in \Z/n\Z})$
 is a
hyperbolic configuration.

We have already proved that the sets $R_i, i\in \Z/n\Z$ are pairwise disjoint.
We will also show that $[b_i^-]_<\cap [b_j^+]_>=\emptyset$ for any
$j\in \Z/n\Z$.  By hypothesis 2., $[b_i^-]_<\cap [b_i^+]_>=\emptyset$, as we are supposing that $f$ is not recurrent. If
$[b_i^-]_<\cap [b_j^+]_>\neq\emptyset$ for some $j\neq i$, then $[b_j^+]_<\subset [b_i^-]_<$, $j\neq i$.   Therefore,
$\overline{\varphi([b_j^+]_<)}
\cap S^1\subset I_i, $  $j\neq i$, which contradicts our hypothesis 4..

We have proved that the sets $R_i$ are disjoint from the sets $A_i, i\in\Z/n\Z$.  So, in order to finish,
we only have to prove that the sets $A_i, i\in \Z/n\Z$ are pairwise disjoint.

If this is not the case, there would exist
$i\neq j$, such that $[b_i^+]_>\cap [b_j^+]_>\neq
\emptyset$.  So,
$[b_i^+]_>\cup [b_j^+]_>$ is a connected set containing $a_{i}\cup a_{j}$, and must therefore
intersect $[b_i^+]_<$, because of the cyclic order and hypothesis 2.  We may of course assume that $ [b_j^+]_>\cap[b_i^+]_<\neq
\emptyset$.  Now, we have that
$[b_j^+]_>$ is a connected set containing $a_{j}\cup a_{i} $ and must therefore intersect $[b_j^+]_<$.  This
contradiction proves our claim.
\end{proof} 

\section{Proof of the main result}\label{proof}

This section is devoted to the proof of Theorem \ref{main*}.

We fix  an
orientation preserving homeomorphism $f : \D \to \D$ which realizes a cycle of links
${\cal L} = ((\alpha_i, \om_i))_{i\in\Z/n\Z}$.  We recall that this means that there exists a family $(z_i)_{i\in \Z/n\Z}$ of points in $\D$ such
that for all $i\in \Z/n\Z$
$$\lim _{k \to -\infty} f^k(z_i) = \alpha _i ,  \ \lim _{k \to
+\infty} f^k(z_i) = \om _i .$$\\

We also recall that $$\ell = \{\alpha_i, \om_i: i\in\Z/n\Z\}\subset S^1,$$ \noindent and that we supppose that $f$ can
 be extended to a  homeomorphism of $\D \cup \ell .$ \\

\subsection{The elliptic case.}

Let us state our first proposition:

\begin{prop}\label{ell} If ${\cal L}$ is elliptic, then $\fix(f)\neq \emptyset$.  Moreover, one of the following holds:

\begin{enumerate}
 \item $f$ is recurrent,
\item  ${\cal L}$ is a degenerate cycle.

\end{enumerate}

\end{prop}

As the proof is long, we will first describe our strategy. The first part of the work consists in constructing a brick
decomposition which is suitable for our purposes.  Once this is done,  we show that if $f$ is not
recurrent, the elliptic order property  gives rise to constraints on the order of the cycle of links
${\cal L}$.  We will show (as a consequence of Lemma \ref{quid}) that the only possibility
for the order of ${\cal L}$ is $n=4$. The case $n=4$ is special, as degeneracies may occur (see Figure 2, and the introduction,
where we explain that non-degeneracy is needed for obtaining the index result).  For $n=4$ we prove that $\fix(f)\neq \emptyset$, and that if $f$ is not recurrent, then
${\cal L}$ is degenerate.\\

{\bf I. Construction of the brick decomposition.}

We first note that we may assume that $n>3$: if $n=3$, the definition
of cycle of links implies automatically that the points $\{\alpha_i\}, \{\om_i\} $ are all different, and the proof follows from Le Calvez's improvement to Handel's
theorem.
As  we are dealing with the
elliptic case, the only possible coincidences among the
points  $\{\alpha _i\}$, $\{\om _i\}$, are of the form $\om_{i-2 }= \alpha _{i}$.
In particular, the points $\{\om _i\}$ are all different and for all $i\in \Z/n\Z$ we can take
a neighbourhood $U_i^+$ of $\om _i$ in $\overline \D$  in such a way that $U_i^+\cap U_j^+ = \emptyset$ if $i\neq j$. We define $U_i^-=U_{i-2}^+$ if $\alpha _i = \om_{i-2}$,
and for all $i\in\Z/n\Z$ such that $\alpha _i \neq \om_{i-2}$ we take a neighbourhood $U_i^-$  of $\alpha_i$ in $\overline \D$  in such a way that
$U_i^-\cap U_j^+ = \emptyset$ for all $j\in\Z/n\Z$ and $U_i^-\cap U_j^- = \emptyset$ for all  $i\neq j$.

{\bf We suppose from now on that $f$ is not recurrent.}

We apply Lemma \ref{ends} and obtain families of closed disks $(b_i'^l)_{l\in \Z\backslash \{0\}, i\in \Z/n\Z}$. So, the disks in
the family $(b_i'^l)_{ l\geq 1, i\in \Z/n\Z}$, have pairwise disjoint interiors.

Let $I_{\mbox{reg}}$ be the set of $i\in \Z/n\Z$ such that $\alpha _i \neq \om _{i-2}$, or such that $\alpha _i =\om _{i-2}$ but there
exists $K>0$ such that $$\cup _{k>K} \inte(b_{i-2}'^k) \cap \cup _{k>K} \inte(b _{i}'^{-k}) = \emptyset.$$
\noindent Let $I_{\mbox{sing}}$ be
the complement of $I_{\mbox{reg}}$ in $\Z/n\Z$.

After discarding a finite number
of disks,  we can suppose that  the disks $b_i'^l$ with $l\geq 1$, $ i\in \Z/n\Z$, and $b_i'^{-l}$ with
$ l\geq 1, i\in I_{\mbox{reg}}$, have pairwise
disjoint interiors.

If $i\in I_{\mbox{sing}}$, then $\alpha _i = \om _{i-2}$  and  for all $k >0$ there exists $k'>k,\  j'>k$, such
that $\inte(b _{i-2}'^{k'}) \cap \inte(b _{i}'^{-j'})\neq \emptyset$. \\ In the following lemma we refer to the
family of integers $(l_i)_{i\in\Z/n\Z}$  constructed in Lemma \ref{ends}.

\begin{lema}\label{cim} If $i\in I_{\mbox{sing}}$,  we can find sequences
of free closed disks  $(c_i^m)_{m\geq 0}$,  such that:

\begin{enumerate}

\item $c_i^m\subset U_{i-2}^+= U_i^-$,
\item there exists an increasing sequence $(k_i^m)_{m\geq 0}$ such that $b_{i-2}'^{k_i^m}\cap c_i^m\neq \emptyset$
for all $m\geq 0$,
\item $(b_{i-2}'^{k_i^p}\cup c_i^p) \cap (b_{i-2}'^{k_i^m}\cup c_i^m) = \emptyset$ for all $p\neq m$,
% \item for all $n\geq 0$, $c'_i^n\cap O = f^{-j_i^n}(z_{i})$, where $O = \{f^k(z_i): k\in \Z, i\in \Z/n\Z\}$,
\item there exists an increasing sequence $(j_i^m)_{m\geq 0}$ such that $f^{-l_i-j_i^m+1}(z_{i}) \in c_i^m$
for all $m\geq 0$,
\item the sequence $(c_i^m)_{m\geq 0}$ converges in the Hausdorff topology to $\om _{i-2} = \alpha _{i}$.
\item $b_{i-2}'^{k_i^m}\cap c_i^m$ is an arc for all $m\geq 0$ (so, $c_i^m \cup b_{i-2}'^{k_i^m}$ is a topological closed disk),
% \item if $c_i^m\cap b_{i-2}'^k \neq \emptyset$, then $c_i^m\cap b_{i-2}'^k$ is an arc,
% \item $b_{i-2}'^{k}\cap c_i^m = \emptyset$ if $k\neq k_i^m$,
\item $\partial (\cup_{k\geq 1}b_{i-2}'^k\cup \cup_{m\geq 0} c_i^m)$ is a one dimensional submanifold,
\item if $x\in \D$, then $x$ belongs to at most two different disks in the family $\{b_{i-2}'^k,c_i^m : k\geq 1,m\geq 0 \}$

\end{enumerate}

\end{lema}

\begin{figure}[h]\label{fcin}
\begin{center}
\psfrag{k}{$\om _{i-2} = \alpha _i$}\psfrag{c}{$b'^{k_m}_{i-2}$}\psfrag{l}{ $c_i^m$}
\includegraphics[scale=0.6]{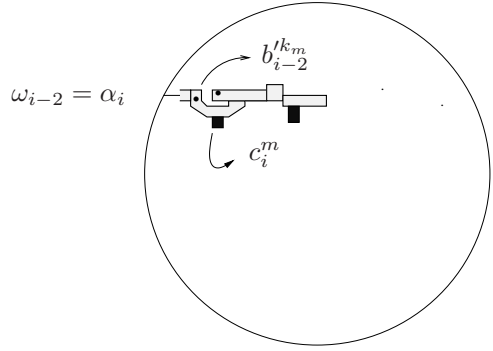}
\caption{The disks $b'^{k_m}_{i-2}$ and $c_i^m$}
\end{center}
\end{figure}

\begin{proof} Take  $i\in I_{\mbox{sing}}$ and consider the family of closed disks $(b_{i-2}'^k)_{k\geq 1} \subset U_{i-2}^+$. As
 $i\in I_{\mbox{sing}}$, there exists $ j_i^0 > 1$, such
that $$\inte(\cup _{k\geq 1}b _{i-2}'^{k}) \cap \inte(b _{i}'^{-j_i^0})\neq \emptyset.$$\noindent By Lemma \ref{ends}, item 7,
$f^{(-l_i-j_i^0+1)} (z_i)\in \inte (b_i'^{-j_i^0})\backslash (\cup_{l\geq 1}b _{i-2}'^{l})$.
We take an arc $$\gamma _i^0 \subset \inte (b_i'^{-j_i^0})\backslash \inte (\cup_{l\geq 1}b _{i-2}'^{l})$$
\noindent joining
$f^{(-l_i-j_i^0+1)} (z_i)$ and a point $x_i^0\in
\partial (\cup _{l\geq 1}b _{i-2}'^{l})$.  We define $k_i^0 \geq 1$ by $$x_i^0\in b _{i-2}'^{k_i^0}.$$

We define inductively for $m\geq 0$:
\begin{enumerate}
 \item $U_m\subset  U_{i-2}^+= U_i^- $ a neighbourhood of $\om_{i-2}=\alpha_i$ in $\overline \D$ such that
$$\overline{U_m}\cap
(\inte(b _{i-2}'^{k_i^m}) \cup \inte(b _{i}'^{-j_i^m})) = \emptyset,$$
\item $K_m >0$ such that for all $k\geq K_m$ $b_{i-2}'^{k}\cup b_{i}'^{-k}\subset U_m$,
\item $j_i^{m+1} > K_m$, such
that $\inte(\cup _{k\geq K_m}b _{i-2}'^{k}) \cap \inte(b _{i}'^{-j_i^{m+1}})\neq \emptyset$,
\item $\gamma_i^{m+1} \subset  \inte (b_i'^{-j_i^{m+1}})\backslash (\cup_{l\geq K_m}b _{i-2}'^{l})$ an arc joining
$f^{(-l_i-j_i^{m+1}+1)} (z_i)$ and a point $x_i^{m+1}\in \partial
(\cup _{k\geq K_m}b _{i-2}'^{k})$,
\item $k_i^{m+1}> K_m$ by $$x_i^{m+1}\in b _{i-2}'^{k_i^{m+1}}.$$

\end{enumerate}

The existence of $K_m$ comes from the fact that both sequences $(b_i'^{-l})_{l\geq 1}$, $(b_{i-2}'^l)_{l\geq 1}$
converge in de Hausdorff topology to $\alpha_i = \om_{i-2}$; that of $ j_i^{m+1}$ from the fact that
$i\in I_{\mbox{sing}}$; that of $\gamma_i^{m+1}$ from the choice of $ j_i^{m+1}$ and the fact that
$f^{(-l_i-j_i^{m+1}+1)} (z_i)\in \inte (b_i'^{-j_i^{m+1}})\backslash (\cup_{l\geq K_m}b _{i-2}'^{l})$, and that
of $x_i^{m+1}$ and $k_i^{m+1}$ follows from the choice of  $ j_i^{m+1}$.

By thickening these arcs $\{\gamma _i^m\}$, we can construct disks $\{c_i^m\}$ verifying all the conditions of the lemma.

\end{proof}

The proposition above allows us to construct a maximal free brick decomposition $(V,E,B)$ such that:

\begin{enumerate}
 \item for all $i\in \Z/n\Z$ and for all $l\geq 1$, there exists $b_i^l\in B$ such that $b_i'^l\subset b_i^l$,
\item for all $i\in I_{\mbox{reg}}$ and for all $l\geq 1$, there exists $b_i^{-l}\in B$ such that $b_i'^{-l}\subset b_i^{-l}$,
\item for all $m\geq 0$ and for all $i\in I_{\mbox{sing}}$ there exists $b_i^{-j_i^m}\in B$ such that $c_i^m\subset b_i^{-j_i^m}$.
\end{enumerate}

{\bf II. The ``domino effect'' of the elliptic order property.}

\begin{lema}\label{dosfut} Take two indices $i,j$ in $\Z/n\Z$, and two integers $k$ and $N$. If $b_j^k$ and
$b_{j+2}^k$ are contained in $[b_i^N]_>$, then there exists $k'\in \Z$ such that  $b_{l}^{k'}$ is contained in $[b_i^N]_>$
for all $l\in\Z/n\Z$.

\end{lema}

\begin{proof} We will show that if $b_j^k$ and $b_{j+2}^k$ are contained in $[b_i^N]_>$, then there exists $k^{''}$
such that both
 $b_{j+1}^{k^{''}}$ and $b_{j+3}^{k^{''}}$ are contained in $[b_i^N]_>$.  If $b_j^k$ and $b_{j+2}^k$ are contained in $[b_i^N]_>$,
$b_j^l$ and $b_{j+2}^l$ are
contained in $[b_i^N]_>$ for all $l\geq k$.  By Remark \ref{puntas}, we can find an arc
$$\gamma : [0,1]\to [b_i^N]_>\cup \{\om_j, \om_{j+2}\}$$\noindent joining $\om_j$ and $\om_{j+2}$.  As $n>3$, and the coincidences
are of the form $\alpha _i = \om _{i-2}$, we know that the points $\alpha _{j+1}, \om_j, \alpha _{j+3}, \om _{j+2}$
are all different.  So, $\gamma $ separates both $\alpha_{j+1}$ from $\om_{j+1}$ and
$\alpha_{j+3}$ from $\om_{j+3}$. So,  there exists $k^{''}>0$ such that
$[b_{j+1}^{k^{''}}]_\leq \cap
[b_i^N]_> \neq \emptyset$
and $[b_{j+3}^{k^{''}}]_\leq\cap [b_i^{N}]_{>} \neq \emptyset$.
We are done by induction, and by taking $k'$ large enough.

\end{proof}

In the following lemma we make reference to the sequences $(k_i^m)_{m\geq 0}$ and $(j_i^m)_{m\geq 0}$ defined in Lemma
\ref{cim}.

\begin{lema}\label{fuckinpuntas} For every $i\in I_{\mbox{sing}}$, there exists $N>0$ such that $[b_i^{-j_i^N}]_\geq$
contains $b_{i-2}^{k_i^N}$.

\end{lema}

\begin{proof}  We will prove the following stronger statement which implies immediately that $[b_i^{-j_i^N}]_\geq$
contains $b_{i-2}^{k_i^N}$: there exists $N>0$ such that  $f(c_i^{N})\cap b_{i-2}'^{k_i^N}\neq \emptyset$.

{\bf I.} Let us begin by studying the local dynamics of the brick decomposition at $\alpha _i = \om _{i-2}$,
$i\in I_{\mbox{sing}}$. We define for all $m\geq 0$, $$X_m = b_{i-2}'^{k_i^m}\cup c_i^{m},$$\noindent and  we recall that every $X_m$
is a closed disk (see Lemma \ref{cim}). Then, for all $m\geq 0$,
% $m>p\geq 0$, $$f^{l_{i-2}+ k_i^p -1}(z_{i-2})\cup f ^{-l_i-j_i^p+1} (z_i)\in X_p$$\noident and
 $$f^{l_{i-2}+ k_i^m -1}(z_{i-2})\cup f ^{-l_i-j_i^m-j_i^m} (z_i)
\in X_m.$$

\noindent So, given any two positive integers $m>p$, one has: $$\cup _{k\geq 1}
f ^k (X_p) \cap X_m \neq \emptyset$$ and $$\cup _{k\geq 1} f ^k (X_m) \cap X_p \neq \emptyset.$$  Besides,
$X_m\cap X_p = \emptyset$ and $X_m$ and $X_p$ are topological closed disks. Therefore, if we can find $m>p\geq 0$ such that
both $X_p$ and $X_m$ are free sets, $f$ would be recurrent by Proposition \ref{guile}. So, we can suppose that for all $m\geq 0$
the set $X_m$ is not free. So, as for all $m\geq 0$ both $b_i'^{k_m}$ and $c_i^m$ are
free sets, then either $f(b_{i-2}'^{k_i^m})\cap  c_i^{m}\neq
 \emptyset$, or $f( c_i^{m})\cap b_{i-2}'^{k_i^m}\neq \emptyset$. If there exists $m>0$ such that $f( c_i^{m})
\cap b_{i-2}'^{k_i^m}\neq \emptyset$, we are done. So, we may assume that for all $m\geq 0$,  $f(b_{i-2}'^{k_i^m})\cap  c_i^{m}\neq\emptyset$. Then,
$f(b_{i-2}^{k_i^m})\cap  b_i^{-j_i^m}\neq \emptyset$ for all $m\geq 0$.
In particular,   $[b_{i-2}^{k_i^m}]_>$ contains $b_{i}^{l}$ for all $l>0$ and for all $m\geq 0$.\\

{\bf II.}  We will show that this implies that $f$ is recurrent. As $[b_{i-2}^{k_i^m}]_>$ contains $b_{i}^{k}$
 and $b_{i-2}^k$, for $k>k_i^m$,
Lemma \ref{dosfut} implies that for all
$m\geq 0$  there exists $l_m>0$ such that $[b_{i-2}^{k_i^m}]_>$ contains $b_{j}^{l}$  for all $j\in \Z/n\Z$ and for all $l\geq
 l_m$.

In particular, Remark \ref{puntas} tells us that for all $m\geq 0$ there exists  an arc
$$\Gamma_m : [0,1]\to  [b_{i-2}^{k_i^m}]_>\cup \{\om _{i-2}, \om _{i-4}\}$$\noindent joining $\om _{i-2}$ and $ \om _{i-4}$, which implies that $\Gamma _m$
separates  $\alpha_{i-1}$ from  $\alpha_{i-3}$ in $\overline \D$ (see Figure 8 (a) and observe that as $n>3$ the points
$\alpha _{i-3}, \om _{i-4},\alpha _{i-1}, \om _{i-2}$ are
all different).  As we are assuming that $f$ is not recurrent, we obtain that the closure of  $[b_{i-2}^{k_i^m}]_\leq$  cannot contain
both points $\alpha_{i-1}$ and $\alpha_{i-3}$.

We will suppose that for all $m\geq 0$, the closure of  $[b_{i-2}^{k_i^m}]_\leq$ does not contain one
of the points $\alpha_{i-1}$ and $\alpha_{i-3}$, and obtain a contradiction. As  $m>p$ implies
$$[b_{i-2}^{k_i^p}]_\leq\subset  [b_{i-2}^{k_i^m}]_\leq , $$ one of the points $\alpha_{i-1}$ or  $\alpha_{i-3}$ is not
contained in the closure of any of the sets $[b_{i-2}^{k_i^m}]_\leq$, $m\geq 0$.
Let us suppose that $\alpha_{i-3}$ is not contained in $\overline{[b_{i-2}^{k_i^m}]_\leq }$ for any $m\geq 0$ (the proof
is analogous in the other case). In particular, for all $m\geq 0$, $[b_{i-2}^{k_i^m}]_\leq$ does not contain any of the bricks containing
the orbit of $z_{i-3}$.  We take a neighbourhood $U$ of $\alpha _{i-3}$ in $\overline \D$ such that $U\cap [b_{i-2}^{k_i^0}]_\leq
=\emptyset$ and such that $U\cap \cup _{l>k_i^0} b_{i-2}^{l} =\emptyset$.  We take $j>0$ such that  $f^{-j}(z_{i-3})\in U$,
and an arc $\beta: [0,1]\to U$ joining $\alpha _{i-3}$ and $f^{-j}(z_{i-3})$. Take a brick $b\in B$ such that
$f^{-j}(z_{i-3})\in b$.  As $\cup _{l\geq 1} b_{i-3}'^{l}\subset [b]_\geq$, Remark \ref{puntas} allows us to take
an arc $\gamma : [0,1]\to [b]_{\geq}\cup \om_{i-3}$ joining $f^{-j}(z_{i-3})$ and $\om_{i-3}$.

So, $\beta . \gamma$
separates $\alpha _{i-2}$ from $\om _{i-2}$ in $\overline \D$ and
 $$\beta . \gamma\cap (\cup _{l>k_0} b_{i-2}^{l} \cup [b_{i-2}^{k_i^0}]_\leq )\neq \emptyset,$$\noindent which implies $$\gamma\cap (\cup _{l>k_0} b_{i-2}^{l} \cup [b_{i-2}^{k_i^0}]_\leq )\neq
\emptyset ,$$\noindent because of our choice of $U$ (see Figure 8 (b)).
So, $$b_{\geq}\cap \cup _{l>0} [b_{i-2}^l]_<\neq \emptyset, $$ which implies that for some $m\geq 0$,
$$[b]_{\geq}\cap [b_{i-2}^m]_<\neq \emptyset.$$

So, $b\in  [b_{i-2}^{k_i^m}]_\leq$, and   $[b_{i-2}^{k_i^m}]_\leq$ contains a brick containing one point of the orbit of
$z_{i-3}$.

This contradiction finishes the proof of the lemma.
\begin{figure}[h]
\begin{center}
\psfrag{a}{$\om _{i-2} = \alpha _i$}\psfrag{b}{$\alpha _{i-1}$}\psfrag{e}{ $\alpha _{i-3}$}\psfrag{d}{ $\om_{i-4}$}
\psfrag{g}{ $\Gamma _m$}\psfrag{x}{ $U$}
   \subfigure[]{\includegraphics[scale=0.6]{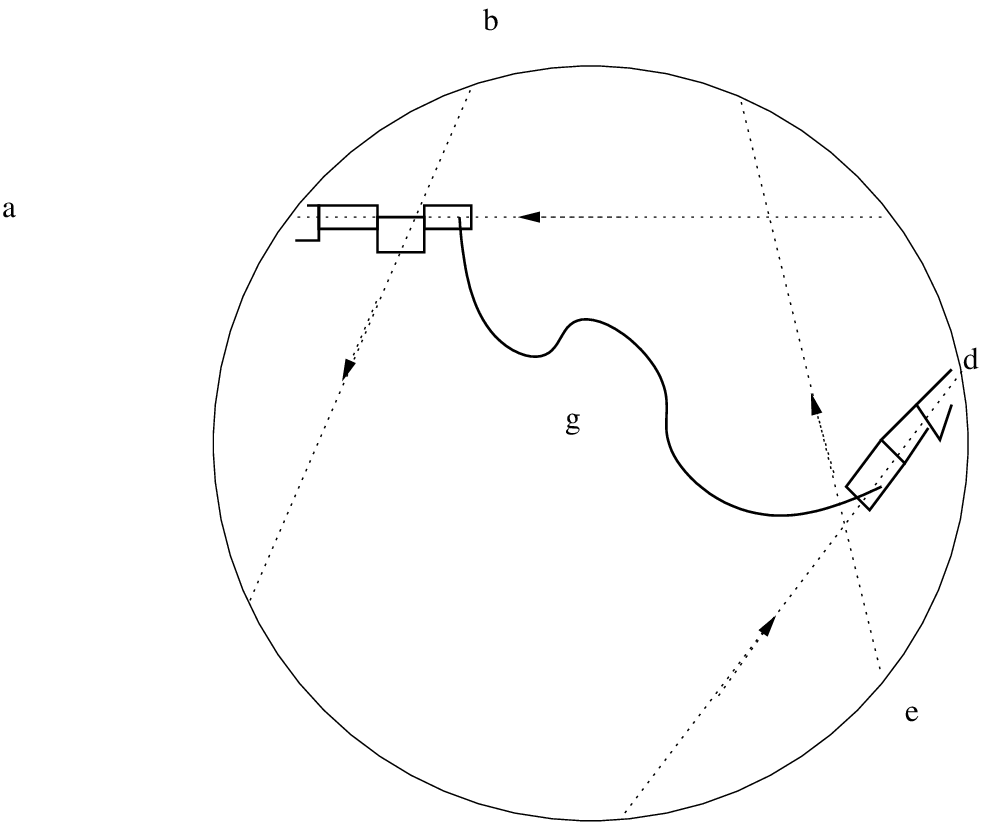}}\hspace{.25in}
    \subfigure[] {\includegraphics[scale=0.6]{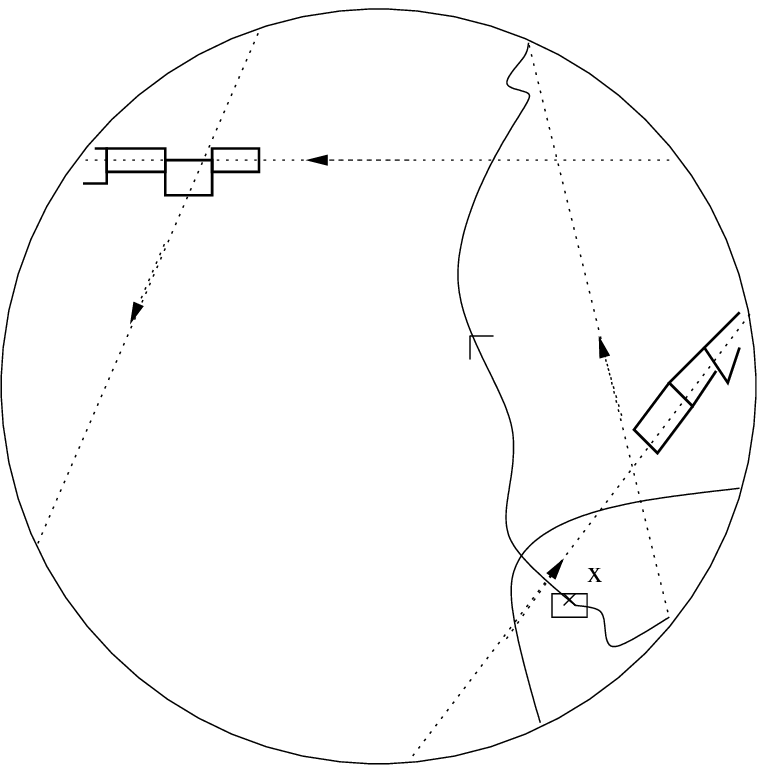}}\hspace{.25in} \\
\caption{The proof of lemma \ref{fuckinpuntas}}
\end{center}
 \end{figure}

\end{proof}

\begin{lema}\label{theka}  There exists $k>0$ such that for any pair of indices $i,j$ in $\Z/n\Z$, the attractor
 $[b_i^{-k}]_>$ contains $b_j^k$.

\end{lema}

\begin{proof} For all $i \in I_{\mbox{reg}}$, we know that $\cup _{l\geq 1} b_i'^{-l}\subset \cup _{l>0} [b_i^{-l}]_>$ (note
that this is not necessarily the case if $i\in I_{\mbox{sing}}$).  So,
by Remark \ref{puntas}, there exists an arc $$\Gamma _i: [0,1] \to \cup _{l>0} [b_i^{-l} ]_>\cup
\{\alpha _i, \om _i\}$$ \noindent joining $\alpha _i $ and  $ \om _i $.  So, $\Gamma _i$ separates both
$\alpha _{i-1}$ from $\om _{i-1}$ and  $\alpha _{i+1}$ from $\om _{i+1}$ in $\overline \D$. Therefore, there
 exists $m >0$ such that
$[b_i^{-m}]_>$ contains both $b_{i+1}^m$ and $b_{i-1}^m$.  By Lemma \ref{dosfut},  $[b_i^{-m}]_>$ contains $b_j^l$ for all
$j\in \Z/n\Z$, and $l$ large enough.

For all $ i \in I_{\mbox{sing}}$, the previous lemma tells us that there exists $N>0$ such that $[b_i^{-j_i^N}]_\geq$ contains
$b_{i-2}^{k_i^N}$.  Clearly, $[b_i^{-j_i^N}]_\geq$ also contains $b_{i}^{k_i^N}$ and so once again, Lemma \ref{dosfut} implies
 that  $[b_i^{-j_i^N}]_\geq$ contains
$b_{j}^{l}$, for all $j\in \Z/n\Z$, and $l$ large enough. We finish by taking $k$ sufficiently large.
\end{proof}

{\bf III. Constraints on the order of the cycle of links ${\cal L}$.}

We fix $k>0$ such that for any pair of indices  $i,j$ in $\Z/n\Z$, $[b_i^{-k}]_>$ contains $b_j^k$.
We define $$a_i = (\cup _{m\geq k} b_i^m) \cap \Gamma _i^+,\  i\in \Z/n\Z$$ \noindent (see Remark \ref{puntas} for the definition of
$\Gamma _i^+$).  We may suppose that $$U = \D \backslash \cup _{i\in\Z/n\Z}
a_i$$\noindent  is simply connected.  As $a_i \subset \cup _{m\geq k} b_i^m$, and we are supposing that $f$ is not recurrent, we know that
$[b_i^{-k}]_<\subset U$ for all $i\in \Z/n\Z$.

Let $\varphi: U \to \D$ be the Riemann map and consider the intervals
$J_i, i\in \Z/n\Z$ defined in 3.1.
We define $I_i$ as to be the connected component of $S^1\backslash \cup _{l\in \Z/n\Z} J_l$ following $J_{i-2}$ in the natural
(positive) cyclic order on $S^1$ .  So, each $I_i$ is a closed interval, and we have:

$$J_{i-2} \to I_{i} \to J_{i-1}$$\noindent  for all $i \in \Z/n\Z.$

\begin{lema}\label{jis}\begin{enumerate} For all $i\in \Z/n\Z$,
             \item  there exists $j_i\in \Z/n\Z$ such that
$\overline {\varphi ([b_i^{-k}]_<)}\cap S^1\subset I_{j_i}$,
\item  $j_i \in \{i-1, i\}$,
\item if $\alpha _i \neq \om _{i-2}$, then $j_i = i$.
            \end{enumerate}

\end{lema}

\begin{proof}
 \begin{enumerate}
  \item If there exists $x\in \overline {\varphi ([b_i^{-k}]_<)}\cap J_{j}$ for some $j\in \Z/n\Z$, then
$\overline{[b_i^{-k}]_<}\cap a_j\neq \emptyset$.  As  $[b_i^{-k}]_<$ is closed in $\D$,
 and as $a_j\subset \D$, we obtain
$[b_i^{-k}]_<\cap a_j\neq \emptyset$, a contradiction.  So, $\overline {\varphi ([b_i^{-k}]_<)}\subset \cup
_{j\in \Z/n\Z} I_j$. If  $\overline {\varphi ([b_i^{-k}]_<)}$ intersects $I_j$ and $I_k$, $k\neq j$, then there exists two different
indices
$i_0$ and $i_1$ in $\Z/n\Z$ such that any arc joining $J_{i_0}$ and $J_{i_1}$ separates $I_j$ from $I_k$.  We take a crosscut
$\gamma$ from $a_{i_0}$ to $a_{i_1}$ such that $\gamma \subset [b_i^{-k}]_>$.  So,
$$\varphi(\gamma \cap U)\cap
\varphi ([b_i^{-k}]_<)\neq \emptyset, $$ and consequently $$[b_i^{-k}]_>\cap [b_i^{-k}]_<\neq \emptyset, $$ which contradicts
our assumption that $f$ is not recurrent.

\item Take a crosscut $\gamma\subset [b_i^{-k}]_>$ from $a_{i-3}$ to $a_{i-1}$.  Then, the elliptic order property implies that
 $\alpha_i$ belongs to the closure
of only one of the two connected components of $U\backslash \gamma$; the one to the right of $\gamma$. We use here
the fact that $\alpha_i \notin \{\om_{i-3}, \om_{i-1}\}$. So, $[b_i^{-k}]_<$ also belongs to the connected component of
$U\backslash \gamma$ which is to the right of $\gamma$.  Consequently, $\varphi([b_i^{-k}]_<)$ belongs to the connected
component of $\D\backslash \varphi(\gamma\cap U)$ which is to the right of $\varphi(\gamma\cap U)$.  As
$\overline{\varphi(\gamma\cap U)}$
is an arc
from $J_{i-3}$ to $J_{i-1}$, the closure of this connected component only contains $I_i$ and $I_{i-1}$.  So,
we obtain $j_i \in \{i-1, i\}$.

\item If $\alpha _i \neq \om _{i-2}$, we can apply exactly the same argument than in the preceding item, but using a crosscut
$\gamma$ from $a_{i-2}$ to $a_{i-1}$, obtaining $j_i = i$.
 \end{enumerate}

\end{proof}

\begin{obs}\label{hipquid}
 If we set $b_i^- = b_i^{-k}$, and $b_i^+ = b_i^k$ , the bricks $ b_{i}^{-}$, $i\in\{i_0, i_1,i_2\}$ satisfy all
the hypothesis of Lemma
\ref{quid}, where $i_0, i_1,i_2$ are any three different indices $\in \Z/n\Z$.  Indeed, $k$ is chosen so
that 2. and 3. (a), hold, 3.(b) is granted since $\alpha_i\subset \overline{[b_i^-]_<}$ for all $i\in \Z/n\Z$, and 3. (c) is the
content of item 1. in the preceding lemma.

\end{obs}

The second item in the preceding lemma gives us:

\begin{cor}\label{koro}  If $|i-l|\geq 2$, then $j_i\neq j_l$.

\end{cor}

The constraints on the order ${\cal L}$ follows.

\begin{lema} The order of ${\cal L}$ is either $4$ or $5$.

\end{lema}

\begin{proof}  If $n\geq 6$, the sets $\{i, i-1\}$, $i\in \{0,2,4\}$ are pairwise disjoint, and so the three indices
$j_0, j_2,j_4$ given by Lemma \ref{jis} are different. This contradicts Lemma \ref{quid}.
\end{proof}

\begin{lema} We have $n=4$.
\end{lema}

\begin{proof}  We show that $n=5$ also contradicts Lemma \ref{quid}. If $j_0, j_2, j_3$ are all different, we are done
because of Lemma \ref{quid}.  Otherwise, the only
possibility is that $j_2 = j_3 = 2$ (see Lemma \ref{jis}). But then, $j_1, j_3$ and $j_4$ are different.

\end{proof}

\begin{lema} ${\cal L}$ is degenerate.

\end{lema}

\begin{proof}
 We will show that if $n=4$ and ${\cal L}$ is non-degenerate, we can also find a triplet $i_0, i_1, i_2$ in $\Z/n\Z$
 such that the corresponding $j_{i_s}$,
$s\in \{0,1,2\}$ are different.

For a non-degenerate cycle of links, there can be at most two coincidences of the type $\alpha _i = \om _{i-2}$.  Furthermore,
if $\alpha _i = \om _{i-2}$ and $\alpha _j = \om _{j-2}$ for some $i\neq j$, then $|i-j|=1$.  Indeed, the points in $\ell$
are ordered as follows:
$$\om_0\toe \alpha_2 \to \om_1 \toe \alpha_3 \to \om_2 \toe \alpha_0\to \om_3\toe \alpha_1\to \om_0,$$
and non-degeneracy means that we cannot have both $\om _{i} =\alpha _{i+2}$ and $\om _{i+2}=\alpha_i$, for some $i\in\Z/4\Z$.
So, there exists $l\in \Z/4\Z$ such that $\alpha _l\neq \om _{l-2}$ and $\alpha _{l+1}\neq \om _{l-1}$.
We can suppose without loss of generality that $\alpha _0 \neq \om _2$, and $\alpha _1 \neq \om _3$ (see Figure 9).
Items 2. and 3. in Lemma \ref{jis} imply that   $j_0, j_1$, and $j_3$ are different, and we are done.

\end{proof}

\begin{figure}[h]
\begin{center}

\psfrag{0}{ $\alpha _0$}\psfrag{1}{ $\alpha _1$}\psfrag{3}{ $\alpha _3$}
\includegraphics[scale=0.6]{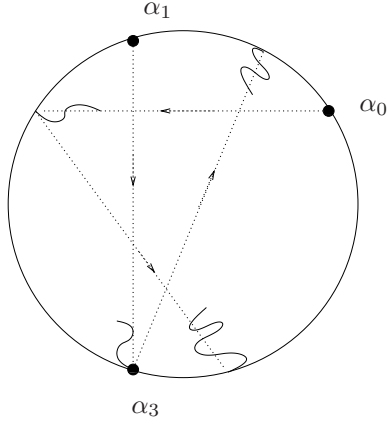}
\caption{The case $n=4$}
\end{center}
\end{figure}

The following lemma finishes the proof of Proposition \ref{ell}.

\begin{lema}  If $n=4$, then $\fix(f)\neq \emptyset$.

\end{lema}

\begin{proof} We will be done by constructing a hyperbolic Repeller/Attractor configuration of order $2$.
We define $$R_0 = [b_0^{-k}]_<, \ R_1 = [b_2^{-k}]_<, \ A_0 = [b_3^{k}]_>, \ A_1 = [b_1^{k}]_> .$$

By the choice of $k$, there exists two bricks $c_i^i, c_i^{i-1}$, contained in $R_i$, $i\in\Z/2\Z$ such that
$[c_i^j]_>\cap A_j\neq \emptyset$, if $j\in \{i, i-1\}$.

Besides, the cyclic order of these sets is the following:

$$R_0\to A_0\to R_1 \to A_1 \to R_0.$$

Indeed, we know that $j_0\in \{0,3\}$, $j_2\in \{2,1\}$, and the cyclic order of the intervals $J_i, I_i, i\in\Z/4\Z$ is:

$$I_0\to J_3\to I_1 \to J_0 \to I_2 \to J_1 \to I_3 \to J_2 \to I_0.$$

So, we just have to show that the sets $R_i, A_i, i\in\Z/2\Z$ are pairwise disjoint.  The choice of $k$ implies that
$[b_i^{-k}]_<\cap[b_j^k]_> = \emptyset$ for all $i,j$ in $\Z/4\Z$.   As a consequence, we just have to check
$R_0\cap R_1 = \emptyset$, and $A_0\cap A_1 = \emptyset$.

If this is not the case, $[b_0^{-k}]_<\cup [b_2^{-k}]_<$ is a connected set separating $[b_1^{k}]_>$ and $ [b_3^{k}]_>$.
Again by the choice of $k$ we have:

$$([b_0^{-k}]_<\cup [b_2^{-k}]_<) \cap[b_0^{-k}]_>\neq \emptyset, $$ \noindent and as we are supposing that $f$ is not
recurrent,
 $$[b_2^{-k}]_<\cap[b_0^{-k}]_>\neq \emptyset.$$ \noindent But then,

 $$[b_2^{-k}]_<\cap [b_2^{-k}]_>\neq \emptyset, $$

\noindent because $[b_2^{-k}]_<$ contains $[b_{0}^{-k}]_< $ and therefore separates $[b_1^{k}]_>$ and $ [b_3^{k}]_>$, both of
which are contained in $[b_2^{-k}]_>$.

Analogously, if $A_0\cap A_1 \neq \emptyset$, then $[b_3^{k}]_>\cup [b_1^{k}]_>$ is a connected set separating $[b_2^{-k}]_<$ and $[b_{0}^{-k}]_< $ .
Again by the choice of $k$ we have:

$$([b_3^{k}]_>\cup [b_1^{k}]_>) \cap[b_3^{k}]_<\neq \emptyset, $$ \noindent and as we are supposing that $f$ is not
recurrent,
 $$[b_1^{k}]_>\cap[b_3^{k}]_<\neq \emptyset.$$ \noindent But then,

 $$[b_1^{k}]_>\cap [b_1^{k}]_<\neq \emptyset, $$

\noindent because $[b_1^{k}]_>$ contains $[b_{3}^{k}]_> $ and therefore separates $[b_0^{-k}]_<$ and $ [b_2^{-k}]_<$, both of
which are contained in $[b_1^{k}]_<$.
\end{proof}

\subsection{The hyperbolic case.}

Our next proposition finishes the proof of Theorem \ref{main*}:

\begin{prop}\label{negi} If ${\cal L}$ is hyperbolic, then $\fix (f) \neq
 \emptyset$.

\end{prop}

We recall that the order of a hyperbolic cycle of links is an even number.  That is, from now on $n=2m$,
$m\geq 2$.
 The
hyperbolic order property implies that the only possible coincidences among the points $\alpha_i,\om_i$, $i\in \Z/n\Z$ are of the form
$\om _{i-2} = \alpha _i$, for even values of $i$,
or $\om_{i+2}= \alpha _i$, for odd values of $i$.

As the points $\{\om _i\}$ are all
different, we can  take a neighbourhood $U_i^+$ of $\om _i$ in $\overline \D$ in such a way that that
$U_i^+\cap U_j^+ = \emptyset$ if $i\neq j$. For even values of $i$, we define $U_i^- = U_{i-2}^+$ if $\alpha_i = \om _{i-2}$,
and if  $\alpha_i \neq \om _{i-2}$ we take a neighbourhood $U_i^-$ of $\alpha _i$ in $\overline \D$ in such a way that
$U_i^-\cap U_j^+ = \emptyset$ for any $j$, and $U_i^- \cap U_j^- = \emptyset$ if $j\neq i$.  Similarly, for odd
values of $i$, we define $U_i^- = U_{i+2}^+$ if $\alpha_i = \om _{i+2}$,
and if  $\alpha_i \neq \om _{i+2}$ we take a neighbourhood $U_i^-$ of $\alpha _i$ in $\overline \D$ in such a way that
$U_i^-\cap U_j^+ = \emptyset$ for any $j$, and $U_i^- \cap U_j^- = \emptyset$ if $j\neq i$.

{\bf We keep the assumption that $f$ is not recurrent.}

We apply Lemma \ref{ends} and obtain families of closed disks $(b_i'^l)_{l\in \Z\backslash \{0\}, i\in \Z/2m\Z}$. So, the disks
in the family
$(b_i'^l)_{l\geq 1, i\in \Z/2m\Z}$ have pairwise disjoint interiors.

Let $I_{\mbox{reg}}$ be the set of even $i\in \Z/2m\Z$ such that $\alpha _i \neq \om _{i-2}$, or such that
$\alpha _i =\om _{i-2}$ but there exists $K>0$ such that $\cup _{k>K}b _{i-2}'^k \cap \cup _{k>K}b _{i}'^{-k} = \emptyset$,
together with the set of odd $i\in \Z/2m\Z$ such that  $\alpha _i \neq \om _{i+2}$, or such that
$\alpha _i =\om _{i+2}$ but there exists $K>0$ such that $\cup _{k>K}b_{i+2}'^k \cap \cup _{k>K}b _{i}'^{-k} = \emptyset$.
Let $I_{\mbox{sing}}$ be the complementary set of $I_{\mbox{reg}}$ in $\Z/2m\Z$.

We can suppose that all the disks in the families $(b_i'^l)_{l\geq 1, i\in \Z/2m\Z}$, $(b_i'^{-l})_{l\geq 1, i\in I_{\mbox{reg}}}$ have
disjoint interiors.

We define  $i^* = i-2$ if $i$ is even, and $i^* = i+2$ if $i$ is odd.

\begin{lema} If $i\in I_{\mbox{sing}}$, we can find sequences of free closed disks  $(c_i^n)_{n\geq 0}$, satisfying :

\begin{enumerate}

\item $c_i^n\subset U_{i^*}^+ = U_i^-$,

\item there exists an increasing sequence $(k_i^n)_{n\geq 0}$
 such that $b_{i^*}'^{k_i^n}\cap c_i^n \neq \emptyset$ for all $n\geq 0$,

\item $(b_{i^*}'^{k_i^n}\cup c_i^n) \cap (b_{i^*}'^{k_i^p}\cup c_i^p) = \emptyset$ for all $n\neq p$,

 \item there exists an increasing sequence $(j_i^n)_{n\geq 0}$ such that $f^{-j_i^n}(z_{i}) \in c_i^n$,

\item the sequence $(c_i^n)_{n\geq 0}$ converge in the Hausdorff topology to $\om _{i^*} = \alpha _{i}$,

\item $b_{i^*}'^{k_i^n}\cap c_i^n$ is an arc for all $n\geq 0$,

\item $\partial (\cup_{k\geq 1}b_{i^*}'^k\cup \cup_{n\geq 0} c_i^n)$ is a one dimensional submanifold,

\item if $x\in \D$, then $x$ belongs to at most two different disks in the family $\{ b_{i^*}'^k, c_i^n: k\geq 1, n\geq 0\}$.
\end{enumerate}

\end{lema}

\begin{proof}  Note that the local dynamics in a neighbourhood of a point $\alpha_i, i\in I_{\mbox{sing}}$ is exactly the same as that
in the elliptic case.  So, the same proof we did for Lemma \ref{cim} works here as well.

\end{proof}

We construct a maximal free brick decomposition $(V,E,B)$ such that:

\begin{enumerate}
 \item for all $i\in \Z/2m\Z$ and for all $l\geq 1$, there exists $b_i^l\in B$ such that $b_i'^l\subset b_i^l$,
\item for all $i\in I_{\mbox{reg}}$ and for all $l\geq 1$, there exists $b_i^{-l}\in B$ such that $b_i'^{-l}\subset b_i^{-l}$,
\item for all $n\geq 0$ and for all $i\in I_{\mbox{sing}}$ there exists $b_i^{-j_i^n}\in B$ such that $c_i^n\subset b_i^{-j_i^n}$.
\end{enumerate}

\begin{lema}\label{l49} If $i\in I_{\mbox{sing}}$, then there exists $N>0$ such that  $[b_i^{-j_i^N}]_\geq$
contains $b_{i^*}^{k_i^N}$.
\end{lema}

\begin{proof}  Fix an even index $i\in I_{\mbox{sing}}$ (the proof for odd indices is analogous).  The first part of the proof
is identical to part I. in the proof of Lemma \ref{fuckinpuntas}. Indeed, this proof is local, that is, it does not depend on how
the rest of the point in $\ell$ are ordered.  So, there are two possibilities:
either $f(c_i^N)\cap b_{i-2}'^{k_i^N}\neq \emptyset$ or $f(b_{i-2}'^{k_i^N})\cap c_i^N\neq \emptyset$.  In the first case
we are done, as it implies immediately the statement of the lemma.  As a consequence, we may assume that for all
$n\geq 0$,  $[b_{i-2}^{k_i^n}]_>$ contains $b_{i}^{l}$ for all $l>0$.  We will show that this contradicts the fact that
$f$ is not recurrent.

With this assumption, for all $n\geq 0$
there exists an arc $$\Gamma _n : [0,1 ] \to [b_{i-2}^{k_i^n}]_>\cup \{\om_{i-2},\om _i\}$$\noindent joining $
\om _{i-2}$ and $\om _{i}$ (see Remark \ref{puntas}).  So, the arc $\Gamma _n$ separates $\alpha _{i-1}$ from
$\alpha _{i-3}$ in $\overline {\D}$ for all $n>0$ (see Figure 10, and note that the points $\alpha _{i-1},
\alpha _{i-3}, \om _{i-2}, \om _{i}$ are all different ).

We deduce
(as we are supposing that $f$ is not recurrent) that for any $n>0$ $\overline{[b_{i-2}^{k_i^n}]_\leq}$ cannot contain
both $\alpha_{i-1}$ and $\alpha _{i-3}$. So, one of the points $\alpha_{i-1}$ or $\alpha _{i-3}$ is not contained
in any of the sets $\overline{[b_{i-2}^{k_i^n}]_\leq}$, $n>0$.  We will suppose that for all $n>0$,
$\alpha _{i-1}\notin\overline{[b_{i-2}^{k_i^n}]_\leq}$ (the proof is analogous in the other case).
We fix $n>0$
and consider the connected set $$K = \cup _{l\geq k_i^n} b_{i-2}^l \cup  [b_{i-2}^{k_i^n}]_\leq .$$  We choose a neighbourhood
$U$ of $\alpha _{i-1}$ in $\overline {\D}$ such that $U\cap K = \emptyset$.  Then, we take $j>0$, such that $f^{-j}(z_{i-1})\in
U$ and $b\in B$ such that $f^{-j}(z_{i-1})\in b$.  We take an arc $\gamma \subset
 U$ joining $\alpha _{i-1}$ and $f^{-j}(z_{i-1})$, and an arc $\beta \subset [b]_\geq \cup \om _{i-1}$ joining $f^{-j}(z_{i-1})$
and $\om _{i-1}$. We deduce that $\gamma .\beta \cap K \neq \emptyset$, and as $\gamma \subset U$, we
have $\beta \cap K\neq \emptyset$. So, there exists $l\geq k_i^n$ such that $b\in [b_{i-2}^l]_{\leq}$, and
consequently $\alpha _{i-1}\in \overline{[b_{i-2}^l]_{\leq}}$.  This contradiction finishes the proof of the lemma.

\end{proof}

\begin{figure}[h]\label{fnueve}
\begin{center}
\psfrag{a}{$\om _{i-2} = \alpha _i$}\psfrag{b}{$\alpha _{i-1}$}\psfrag{c}{$\alpha _{i-2}$}\psfrag{d}{ $\om_{i}$}
\psfrag{g}{ $\Gamma _n$}\psfrag{y}{ $\om_{i-1}$}\psfrag{x}{ $\om _{i+1}$}\psfrag{z}{ $\alpha _{i+1}$}
\includegraphics[scale=0.6]{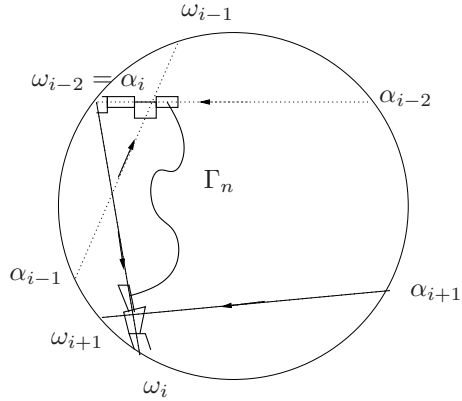}
\caption{The proof of lemma \ref{l49}}
\end{center}
\end{figure}

\begin{lema}\label{ptasfut}  There exists $k>0$ such that for all even values of $i\in \Z/2m\Z$,
both attractors $[b_i^{-k}]_>$ and
$[b_{i-1}^{-k}]_>$
contain $b^k_l$ for all $l\in \{i-2, i-1, i, i+1\}$.

\end{lema}

\begin{proof} If $i\in I_{\mbox{sing}}$, the previous lemma tells us that there exists $N>0$ such that $[b_i^{-j_i^N}]_\geq$ contains
$b_{i-2}^{k_i^N}$.  So, we can find an arc $$\Gamma: [0,1]\to
[b_i^{-j_i^N}]_>\cup \{\om _{i-2}, \om _{i}\}$$\noindent  joining $\om _{i-2}$ and $\om _i$.  This arc
separates both $\alpha _{i-1}$ from $\om _{i-1}$, and $\alpha _{i+1}$ from $\om _{i+1}$ in $\overline {\D}$ (see Figure 10).  As a consequence,
both $\cup _{k\geq 1} [b_{i-1}^k]_\leq$ and $\cup _{k\geq 1} [b_{i+1}^k]_\leq$ intersect $\Gamma$, and so there exists $k>0$
such that $b_{i-1}^k$ and $b_{i+1}^k$ belong to  $[b_i^{-j_i^N}]_>$.  If $i-1\in I_{\mbox{sing}}$, we can show analogously
that
 $ [b_{i-1}^{-j_{i-1}^N}]_>$
contains $b^k_l$ for all $l\in \{i-2, i-1, i, i+1\}$ and some $k>0$.

If $i\in I_{\mbox{reg}}$, we can find an arc $$\Gamma : [0,1] \to \cup _{l>0}[b_i^{-l}]_>\cup \{\alpha _i, \om_i\}$$\noindent joining $\alpha_i$ and
$\om_i$.  So, $\Gamma$ separates (in $\overline{\D}$) both $\alpha _{i+1}$ from $\om _{i+1}$ and
 $\alpha _{i-1}$ from $\om _{i-1}$. So,  both $\cup _{k\geq 1} [b_{i-1}^k]_\leq$ and $\cup _{k\geq 1} [b_{i+1}^k]_\leq$
intersect
$\Gamma$, and there exist $k,N>0$ such that $[b_i^{-N}]_> \cap [b_{i-1}^k]_\leq\neq \emptyset$ and $[b_i^{-N}]_> \cap
[b_{i+1}^k]_\leq\neq \emptyset$. Once $b_{i-1}^l$ and $b_{i+1}^l$ belong to $[b_i^{-N}]_>$,  we can find an arc $$\Gamma ': [0,1]
\to [b_i^{-N}]_> \cup \{\om _{i-1}, \om_{i+1}\}$$\noindent joining $\om_{i-1}$ and $\om_{i+1}$. So, $\Gamma '$ separates
$\alpha _{i-2}$ from $\om_{i-2}$ in $\overline \D$, and one obtains $b_{i-2}^k\in [b_i^{-N}]_> $, for some $k>0$. We obtain
the result by sufficiently enlarging $k$.
\end{proof}

We fix $k>0$ as in Lemma \ref{ptasfut}.

\begin{lema}\label{thel}  There exists $p>k$ such that $[b_i^{-k}]_<\cap b_{j}'^l = \emptyset$ for all $i,j$ in $\Z/2m\Z$
 and $l\geq p$.

\end{lema}

\begin{proof} Fix  $i\in \Z/2m\Z$ even. There exists an arc
$$\gamma _{i}:[0,1]\to [b_i^{-k}]_>\cup \{\om_{i+1}, \om_{i-1}\}$$\noindent joining $\om_{i+1}$ and $\om_{i-1}$.  As the
three points $\alpha_{i}, \om_{i+1},$ and $\om_{i-1}$ are different, $\gamma_{i}$ separates $\alpha_{i}$ from any $\om_j$
$j\notin \{i-2, i-1,i+1\}$ (in  $\overline \D$) .

So, there exists $l_i>k$ such that  $\gamma_{i}$ separates $[b_i^{-k}]_<$ from any $b_{j}'^l$ with $l>l_i$ and
$j\notin \{i-2,i-1,i+1\}$.    Besides,
we already know that  $[b_{i}^{-l_i}]_{<}\cap [b_{j}^{l_i}]_> = \emptyset$ if $j\in \{i-2,i-1,i+1\}$, because $[b_{i}^{-l_i}]_>$
contains $b_j^{l_i}$. In particular, $[b_{i}^{-l_i}]_<\cap b_{j}'^{l} = \emptyset$ for $l\geq l_i$ and $j\in \{i-2,i-1,i+1\}$.

If $i$ is odd, we can do the same argument  with an arc
$$\gamma _{i-1}:[0,1]\to [b_i^{-k}]_>\cup \{\om_{i}, \om_{i-2}\}$$\noindent joining $\om_{i}$ and $\om_{i-2}$.

We finish by taking $p=\max\{l_i, i\in\Z/2m\Z \}$.

\end{proof}

Thanks to the two preceeding lemmas we may fix $k>0$ such that:
\begin{enumerate}
 \item both attractors $[b_i^{-k}]_>$ and $[b_{i-1}^{-k}]_>$
contains $b^k_l$ for all even values of $i$, and for all $l\in \{i-2, i-1, i, i+1\}$,
\item $[b_i^{-k}]_<\cap b_{j}'^l= \emptyset$ for all $i,j $ in $\Z/2m\Z$, and $l\geq k$.
\end{enumerate}

We define $$a_i = \Gamma _i^+\cap \cup_{l\geq k} b_{i}'^{l}$$\noindent for all $i\in \Z/2m\Z$.
The cyclic order of the sets $\{a_i\}$ satisfies:  $$a_{i-2}\to a _{i+1}\to a_{i},$$\noindent  for all even values
of $i$.  We
may suppose that each $a_i$ is an arc, and so $U= \D\backslash \cup_{i\in\Z/2m\Z}a_i$ is simply connected.  Let $\varphi: U \to \D$ be the Riemann map
and consider the intervals $\{J_i\}$ defined in 3.1.

For all even $i$, we define $I_i$ as to be the connected component  of
$S^1\backslash \cup _{l\in \Z/2m\Z}J_l$ following $J_{i-2}$ in the natural
(positive) cyclic order on $S^1$. We define $I_{i+1}$,  as to be the connected component of
$S^1\backslash \cup _{l\in \Z/2m\Z}J_l$ following $I_i$.  So, for all even $i$ we have:

$$J_{i-2}\to I_i\to J_{i+1}\to I_{i+1}\to J_i.$$

\begin{lema}\label{hypqh}  For all $i\in \Z/2m\Z$,

\begin{enumerate}
 \item $[b_i^{-k}]_<\subset U$,
\item if $i$ is even, then $\overline{\varphi([b_i^{-k}]_<)}\cap S^1\subset I_{i}\cup I_{i-1}$
, and
$\overline{\varphi(b_{i-1_<}^{-k})}\cap S^1\subset I_{i}\cup I_{i+1}$,
\item there exists $j_i$ such that  $\overline{\varphi([b_i^{-k}]_<)}\cap S^1\subset I_{j_i}$
(so, if $i$ is even, $j_i\in \{i, i-1\}$, $j_{i-1}\in \{i, i+1\}$).
\end{enumerate}

\end{lema}

\begin{proof}
\begin{enumerate}
 \item This is trivial because of the choice of $k>0$.

\item First, we show that $\overline {\varphi ([b_i^{-k}]_<)}\subset \cup
_{j\in \Z/2m\Z} I_j$.  Otherwise, there exists $x\in \overline {\varphi ([b_i^{-k}]_<)}\cap J_{j}$ for
some $j\in \Z/2m\Z$.  So, $\overline{[b_i^{-k}]_<}$
contains a point  in $a_j$.  As  $[b_i^{-k}]_<$ is a closed subset of $\D$, and $a_j\subset \D$
we obtain $[b_i^{-k}]_<\cap a_j\neq \emptyset$, contradicting the previous item.

Fix if $i\in\Z/2m\Z$ even.  Take a crosscut $\gamma \subset [b_i^{-k}]_>$ from $\om_{i-1}$ to $\om_{i+1}$.  So, $\alpha _i$ belongs
to the closure of only one of the connected components of $\overline \D\backslash \gamma$; the one to the right of $\gamma$.
So, $\varphi ([b_i^{-k}]_<)$ belongs to the connected component of $\D\backslash \varphi(\gamma\cap U)$ which is to the
right of
$\varphi(\gamma\cap U)$.  As $\overline{\varphi(\gamma\cap U)}$ is an arc joining $J_{i-1}$ and $J_{i+1}$, the cyclic order implies that
$\overline {\varphi ([b_i^{-k}]_<)}\cap S^1\subset I_{i}\cup I_{i-1}$.

The statement for $i-1$ is proved analogously.

\item  Suppose $i$ is even (as before, the other case is analogous). The previous item implies that if
$\overline{\varphi([b_i^{-k}]_<)}$ intersects $I_j$ and $I_l$, $j\neq l$, then $\{j,l\}=\{i, i-1\}$.

Take a crosscut $\gamma \subset [b_i^{-k}]_>$ from $\om_{i-1}$ to $\om_{i-2}$.  Then, $\overline{\varphi(\gamma\cap U)}$
 separates
in $\overline {\D}$ $I_{i-1}$ from $I_i$.  This gives us $$[b_i^{-k}]_<\cap [b_i^{-k}]_>\neq \emptyset,$$\noindent a
contradiction.
\end{enumerate}

\end{proof}

\begin{obs}  If we set $a'_i = a_{2i},$ $b_i^- = b_{2i}^{-k}$, and $b_i^+= b_{2i}^k$ for all $i\in \Z/m\Z$, then $a'_i$,
$b_i^-$, $b_i^+$, $i\in \Z/m\Z$,
satisfy hypothesis 1. to 3. of Lemma \ref{quidh}.  So, if we prove that $j_{2i}=2i$ for all $i\in\Z/m\Z$, then
$\fix(f)\neq \emptyset$.  Indeed, the sets  $a'_i, i\in\Z/m\Z$ are cyclically ordered as follows:

$$a'_0\to a'_1 \to a'_2\to \ldots \to a'_{m-2}\to a '_{m-1} \to a'_0.$$ \noindent Besides, if we set $J'_i = J_{2i},$
for all $i\in \Z/m\Z$, we have:

$$J'_{i-1}\to I_{2i} \to J'_i,$$ \noindent for all $i\in\Z/2m\Z$, and so  $j_{2i}=2i$ is exactly hypothesis 4. of
Lemma \ref{quidh}.

\end{obs}

We are now ready to prove Proposition \ref{negi}:

\begin{proof}  Because of the previous remark, it is enough to show that $j_{2i}=2i$ for all $i\in\Z/m\Z$. We will show
that if this is not the case, we contradict Lemma \ref{quid}.
Lemma \ref{hypqh}, tells us that $j_{2i}\in \{2i, 2i-1\}$.  Let us assume that $j_{2i} = 2i-1$. This implies that
$j_{2i-2}, j_{2i-1},$ and $j_{2i}$ are different.
Indeed, by Lemma \ref{hypqh} $j_{2i-2}\in \{2i-3, 2i-2\}$, $j_{2i-1}\in \{2i, 2i+1\}$, and by assumption $j_{2i} = 2i-1$.  Besides,
we have:

\begin{itemize}

\item $[b_{2i}^{-k}]_>$ contains $b_{2i}^k$, $b_{2i-1}^k$, and  $b_{2i-2}^k$,
\item $[b_{2i-1}^{-k}]_>$ contains $b_{2i}^k$, $b_{2i-1}^k$, and  $b_{2i-2}^k$,
\item  $[b_{2i-2}^{-k}]_>$ contains both $b_{2i-2}^k$ and $b_{2i-1}^k$.
\end{itemize}

So, as $j_{2i-2}, j_{2i-1},$ and $j_{2i}$ are different, if we show that $[b_{2i-2}^{-k}]_>$ also contains $b_{2i}^k$, we
contradict Lemma \ref{quid}.
Take a crosscut $\gamma \subset  [b_{2i-2}^{-k}]_> $ from $a_{2i-2}$ to $a_{2i-4}$.  Then,
$\overline{\varphi (\gamma \cap U)}$ separates $I_{2i-1}$ from $J_{2i}$.  On the other hand, $\overline{\varphi([b_{2i}^{k}]_<)}$
joins this both sets, as we are assuming $j_{2i} = 2i-1$, and by definition of $J_{2i}$. So,
$$\varphi([b_{2i}^{k}]_<)\cap \varphi (\gamma \cap U)\neq
\emptyset, $$\noindent and we are done.

\end{proof}

\section{Proof of Lemma \ref{opt}}\label{last}

We finish by proving Lemma \ref{opt}, showing that our theorem is optimal.  \\

We begin with a perturbation lemma.\\

Let $(\phi _t)_{t\in \R}$ be the flow in $\D$ whose orbits are drawn in the figure below:

\begin{figure}[h]
\begin{center}
\psfrag{a}{$0$}
{\includegraphics[scale=0.6]{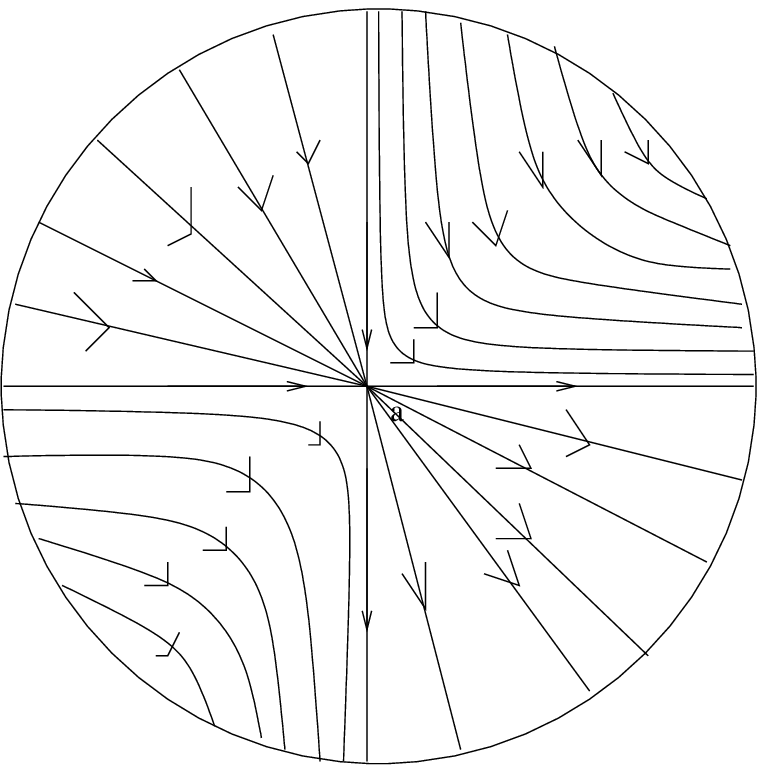}}

\end{center}
\end{figure}

We say that a flow $(\varphi _t)_{t\in \R}$ in $\D$ is {\it locally conjugate to $(\phi _t)_{t\in \R}$ at $z_0$} if there exist an open neighbourhood $U$ of $z_0$ and a homeomorphism
$h:\D\to U$ such that $h(0)=z_0$ and $h^{-1}\varphi_t h = \phi_t$ for all $t\in \R$. 

If $\varphi:\D \to \D$ is a homeomorphism, we write $\alpha(x, \varphi)$ for the set of accumulation points of the backward $\varphi$- orbit of $x$, and $\om (x,\varphi)$
for the set of accumulation points of the forward $\varphi$- orbit of $x$.

\begin{lema}\label{local} Let $\varphi:\D \to \D$ be the time one map of flow which is locally conjugate to $(\phi_t)_{t\in \R}$ at $z_0$, and $U$ an open neighbourhood 
of $z_0$ where
$h^{-1}\varphi h = \phi _1$. Then, for any $x,y \in U$ such that $\om (x, \varphi) = z_0 =\alpha(y, \varphi) $,
there exists an orientation preserving homeomorphism $g: \D\to \D$ supported in the union of two free disjoint open disks such that
$$ \alpha (x, \varphi\circ g) = \alpha (x, \varphi), \ \om (x, \varphi\circ g) = \om (y, \varphi) .$$
 
\end{lema}

\begin{proof} Let $\Delta \subset \D $ be the straight oriented line through $0$ with tangent unit vector $e^{i\pi/4}$, and let $L$ (resp. $R$) be the connected 
component of $U\backslash h(\Delta)$ which is to the left 
(resp. the right) of $h(\Delta)$. 

Note that given two points $z_1, z_2$ in the same connected component $C$ of $U\backslash h(\Delta)$
 that do not
belong to the same orbit of $(\varphi_t)_{t\in \R}$ there exists an arc $\delta \subset C$ joining $z_0$ and $z_1$ such that $\varphi(\delta)\cap \delta = \emptyset$.  Besides, any $x\in U$ 
such that $\om (x,\varphi) = z_0$ belongs to $L$, and any $y\in U$ 
such that $\alpha (y,\varphi) = z_0$ belongs to $R$.  Moreover, there 
exist $z\in L$  and $n>0$ such that $\varphi ^n (z) \in R$. 

So, we can take a free arc $\delta_1\subset L$ joining $x$ and $z$ and a free arc $\delta _2\subset R$ joining  $\varphi ^n (z)$ and $\varphi ^{-1}(y)$.
Moreover, we may suppose that $$\delta_1\cap  \{\varphi^{-k}(x): k>0\} =  \delta_2\cap  \{\varphi^{k}(y): k\geq 0\} =
 (\delta_1\cup \delta_2) \cap  \{\varphi^{k}(z): 0<k<n\} = \emptyset.$$
We thicken the $\delta_i$'s
into open free  and disjoint disks $D_1\subset L$, $D_2\subset R$, such that $$D_1\cap  \{\varphi^{-k}(x): k>0\} = D_2\cap  \{\varphi^{k}(y): k\geq 0\} =
(D_1\cup D_2)\cap  \{\varphi^{k}(z): 0<k<n\} = \emptyset.$$

Finally, we construct an orientation preserving homeomorphism $g: \D\to \D$ supported in $D_1 \cup D_2$ such that 
$g(x) = z$ and $g(\varphi ^n (z)) = \varphi^{-1}(y)$.  We obtain $$ \alpha (x, \varphi\circ g) = \alpha (x, \varphi), \ \om (x, \varphi\circ g) = \om (y, \varphi) ,$$
 \noindent as we wanted.
 
\end{proof}

\begin{obs}  In fact, given a finite set of points $x_i,y_i \in U, i=1, \ldots, n$ which belong to different orbits of  $(\varphi_t)_{t\in \R}$ and
such that $\om (x_i) = z_0=\alpha(y_i)$, $i=1, \ldots, n$,
there exists an orientation preserving homeomorphism $g: \D\to \D$ supported in a finite union of free disjoint open disks such that
$$ \alpha (x_i, \varphi\circ g) = \alpha (x_i, \varphi), \ \om (x_i, \varphi\circ g) = \om (y_i, \varphi) ,$$ \noindent $i=1, \ldots, n$. Indeed, we choose different
points $z_i\in L$ and positive integers $n_i>0$ such that $\varphi^{n_i}(z_i)\in R$.  Then, we take pairwise disjoint arcs $\delta_i^1$ joining $x_i$ and $z_i$ and  
$\delta_i^2$ joining $\varphi^{n_i}(z_i)$ and $\varphi^{-1}(y_i)$ in such a way that all these arcs are disjoint from the backward $\varphi$-orbit of $x_i$, the forward
$\varphi$-orbit of $y_i$ and the transitional orbits $\varphi(z_i), \ldots, \varphi^{n_i-1}(z_i)$. This allows us to construct the desired perturbation $g$.

\end{obs}

Given a 
family ${\cal K}=((\alpha_i, \om_i))_{i\in\Z/n\Z}$ of pairs of points in $S^1$, we note $\Delta _i$  the oriented segment joining $\alpha _i$ 
and $\om _i$. We say that $z\in \D$ is a {\it multiple point}  if $z$ belongs to at least two different $\Delta_i$'s .  Let $z$ be a multiple
point, and let $I= \{i\in \Z/n\Z: z\in \Delta _i\}$.  We say that a 
multiple point $z\in \D$ has {\it zero-index} if there exists a straight oriented line $\Delta$ containing $z$ such that the algebraic intersection
number $\Delta \wedge \Delta _i =1$
for all $i\in I$. Note that this is the case for any multiple point such that $\# I = 2$.

We say that a pair  $(\alpha_k, \om_k)\in {\cal K}$ is {\it  $i$-separated } if $\alpha _k$ and $\om_k$ belong to different connected components of
$S^1\backslash \{\alpha_i, \om _i\}$ .

A {\it degeneracy} of ${\cal K}$ is a pair of elements of the family $(\alpha_i, \om_i)$ and $(\alpha_j, \om_j)$ such that $\alpha_j = \om_i$
and  $\alpha_i = \om_j$.  We say that a degeneracy is {\it trivial} if the following holds: 
the connected component of $S^1\backslash \{\alpha_i, \om _i\}$ containing $\alpha _k$ is independent of the $i$-separated pair
 $(\alpha_k, \om_k)\in {\cal K}$. 

We will deduce Lemma \ref{opt} from the following lemma.

\begin{lema}  Let  ${\cal K} = ((\alpha_i, \om_i))_{i\in\Z/n\Z}$ be a 
family of pairs of points in $S^1$.  We suppose that:

\begin{enumerate}
\item every multiple point is of zero index;
\item  every polygon $P\subset \D$ whose boundary is contained in $\cup_{i\in \Z/n\Z} \Delta _i$ has zero index,
\item every degeneracy is trivial.
\end{enumerate}

Then, there exists a flow $(\varphi _t)_{t\in \R}$ in $\D$ such that:

\begin{enumerate}
 \item $(\varphi _t)_{t\in \R}$ is locally conjugate to $(\phi _t)_{t\in \R}$ at every singularity $z_0$;
\item for all $i\in \Z/n\Z$ there exist two points $z_i^-, z_i^+ \in \D$ such that $\alpha (z_i^-) = \alpha_i$ and $\om(z_i^+) = \om_i$;
\item the $2n$ points $z_i^-, z_i^+ $, $i\in \Z/n\Z$ are different.
\end{enumerate}

\end{lema}

\begin{proof}  First suppose that there are no degeneracies in ${\cal K}$. In this case, the orientations of the $\Delta _i$'s induce a flow $(\varphi _t)_{t\in \R}$ on 
$\cup _{i\in \Z/n\Z} \Delta _i$ with a singularity at each multiple point.  By hypothesis 1., we may extend this flow to a neighbourhood of every multiple point in 
such a way that it is locally conjugate to $(\phi_t)_{t\in \R}$. Moreover, by hypothesis 2. we may extend $(\varphi _t)_{t\in \R}$ to the rest of $\D$ 
without singularities, and we are done.

If ${\cal K}$ contains one degeneracy $(\alpha_i, \om_i)=(\om_j, \alpha_j)$, we ``open it up'' as follows. We consider the family of segments
$\cup _{k\in\Z/n\Z, k\neq j} \Delta _k$ and a simple curve $\gamma _j$ joining $\alpha _j$ and $\om _j$ such that:

\begin{enumerate}
 \item $\gamma _j \cap \Delta _i = \{\alpha_i, \om_i\}$,
\item $\gamma _j \cap \Delta _k \cap \D \neq \emptyset$ if and only if $(\alpha_k, \om_k)$ is $j$- separated, and in this case  
$\#\{\gamma _j \cap \Delta _k \cap \D \} = 1$,
\item $\gamma _j$ does not intersect any multiple point.
\end{enumerate}

Now, the orientations of the $\Delta _i$'s  $i\neq j$, and the orientation of $\gamma _j$ induce a flow $(\varphi _t)_{t\in \R}$ on 
$\cup _{i\in \Z/n\Z, i\neq j} \Delta _i \cup \gamma_j$ with a singularity at each multiple point of $\cup _{i\in \Z/n\Z, i\neq j} \Delta _i$ and also
at the intersection points of $\gamma _j$ with the $\Delta_i$'s, $i\neq j$.

Note that as $\gamma _j$ does not intersect any multiple point, we may extend $(\varphi _t)_{t\in \R}$ to a neighbourhood of every multiple point of 
$\cup _{k\in\Z/n\Z, k\neq j} \Delta _k$
in 
such a way that it is locally conjugate to $(\phi_t)_{t\in \R}$. Moreover, a point 
$z_0\in \gamma _j$ belongs to at most one $ \Delta _k$, $k\neq j$, and the intersection is transversal by item 2. above.  So,
we may as well extend $(\varphi _t)_{t\in \R}$ to a neighbourhood of $z_0$ so as to have local conjugation with $(\phi_t)_{t\in \R}$ as well.
As degeneracies are trivial, we can extend $(\varphi _t)_{t\in \R}$ to the rest of $\D$ without singularities.

If more than one degeneracy occurs, triviality implies that they are disjoint.  That is, if $(\alpha_i, \om_i)=(\om_j, \alpha_j)$,
and $(\alpha_k, \om_k)=(\om_l, \alpha_l)$, then $(\alpha_i, \om_i)$ is not $k$-separated. So, we can ``open up'' both degeneracies
in such a way that $\gamma _j\cap \gamma _l = \emptyset$, and construct our flow $(\varphi _t)_{t\in \R}$ analogously. 
\end{proof}

We deduce:

\begin{cor} With the same hypothesis of the preceeding lemma, there exists a fixed-point free orientation preserving homeomorphism
$f:\D \to \D$ that realizes ${\cal K}$.
 
\end{cor}

\begin{proof}
 Let $\varphi$ be the time one map of the flow given by the preceeding lemma.  By simultaneous applications of Lemma \ref{local}, we can
construct 
an orientation preserving homeomorphism $g : \D \to \D$ supported in disjoint open free disks such that 

$$\lim _{k\to -\infty} (\varphi \circ g)^k (z_i^-) = \alpha _i, \ 
\lim _{k\to \infty} (\varphi \circ g)^k (z_i^-) = \om _i ,$$
\noindent (see as well the remark following Lemma \ref{local}). 

Then, the homeomorphism $\varphi \circ g$ realizes ${\cal K}$.  Moreover, as we have local conjugation to  the flow $(\phi_t)_{t\in R}$ at every singularity of $\varphi$,
and 
$\varphi \circ g = \varphi$ in a neighbourhood of each singularity, we can further perturb $\varphi \circ g$ into a homeomorphism $f:\D \to \D$ realizing ${\cal K}$
and which is fixed point free.

\end{proof}

This last lemma finishes the proof of Lemma \ref{opt}:

\begin{lema}  If a multiple point has non-zero index, then there exists a subfamily of ${\cal K}$ forming an elliptic cycle of links.
 
\end{lema}

\begin{proof}  Let $x$ be a multiple point of non zero index, and let $I = \{i\in \Z/n\Z: x\in \Delta _i\}$.  As $x$ has non-zero index, there exists indices $i,j\in I$ 
such that the oriented
interval in $S^1$ joining $\alpha _i$ and $\alpha _j$ contains $\om _k$, $k\in I$.  Then, ${\cal L} = (\alpha '_l, \om '_l)_{l\in \Z/3\Z}$ is an elliptic
cycle of links, where $(\alpha '_0, \om'_0) = (\alpha _i, \om_i)$, $(\alpha '_1, \om'_1) = (\alpha _j, \om_j)$,  and $(\alpha '_2, \om'_2) = (\alpha _k, \om_k)$.

\end{proof}

\section{Acknowledgements}

Professor Patrice Le Calvez carefully read and corrected this article and I thank him deeply. 

I also thank Professors Christian Bonatti and Fr\'ed\'eric Le Roux who highlighted the importance of having an optimal theorem; their remarks gave rise
to Lemma \ref{opt}.

\bibliography{jules}
\bibliographystyle{plain}

\author{$\ $ \\
Juliana Xavier\\
  I.M.E.R.L,\\
  Facultad de Ingenier\'ia,\\
Universidad de la Rep\'ublica,\\
  Julio Herrera y Reissig,\\
  Montevideo, Uruguay.\\
  \texttt{mariajules@gmail.com}}

\end{document}